\documentclass[11pt]{article}
\usepackage[numbers,sort&compress]{natbib}
\usepackage{appendix,diagbox,booktabs,float}
\usepackage{enumerate}
\usepackage{amscd}
\usepackage{amsmath}
\usepackage{latexsym}
\usepackage{amsfonts}
\usepackage{amssymb}
\usepackage{amsthm}
\usepackage{bm}
\usepackage{verbatim}
\usepackage{mathrsfs}
\usepackage{enumerate}
\usepackage{hyperref}

\theoremstyle{plain}
\theoremstyle{definition}\newtheorem{theorem}{Theorem}[section]
\theoremstyle{plain}\newtheorem{lemma}[theorem]{Lemma}
\theoremstyle{plain}\newtheorem{coro}[theorem]{Corollary}
\theoremstyle{plain}
\theoremstyle{remark}\newtheorem{remark}{Remark}[section]
\usepackage{xcolor}

\newcommand{\Div}{\mathrm{div}\,}
\newcommand{\B}{\Big}

\newcommand{\s}{\mathrm{div}}

\newcommand{\be}{\begin{equation}}
\newcommand{\ee}{\end{equation}}
 \newcommand{\ba}{\begin{aligned}}
 \newcommand{\ea}{\end{aligned}}

  \newcommand{\f}{\frac}
    
  \newcommand{\ben}{\begin{enumerate}}
   \newcommand{\een}{\end{enumerate}}
\newcommand{\bv}{{\bm v}}
\newcommand{\bx}{{\bm x}}

\newcommand{\varep}{\varepsilon}

\newcommand{\bl}{{\mbox{\boldmath $\ell$}}}
\newcommand{\bn}{\hat{{\mbox{\boldmath $\ell$}}}}
\newcommand{\bomega}{{\mbox{\boldmath $\omega$}}}

\newcommand{\Rmnum}[1]{\expandafter\@slowromancap\romannumeral #1@}

\allowdisplaybreaks

\numberwithin{equation}{section}
\begin{document}
%%%%%%%%%%%%%%%%%%%%%%%%%%%%%%%%%%%%%%%%%%%%%%%%%%%%%%%%%%%%%%%%%%%%%%%%%%%%%%%%%%%%%%%%%%%%%%%%%%%%
\title{Four-fifths   laws   in     electron and Hall
magnetohydrodynamic fluids: Energy, Magnetic helicity and Generalized helicity
 }
\author{Yanqing Wang\footnote{   College of Mathematics and   Information Science, Zhengzhou University of Light Industry, Zhengzhou, Henan  450002,  P. R. China Email: wangyanqing20056@gmail.com},   ~  %\,
   Yulin Ye\footnote{Corresponding author. School of Mathematics and Statistics,
		Henan University,
		Kaifeng, 475004,
		P. R. China. Email: ylye@vip.henu.edu.cn}~   and  Otto Chkhetiani\footnote{A. M. Obukhov Institute of Atmospheric Physics, Russian Academy of Sciences, Pyzhevsky per. 3, Moscow, 119017, Russia.  Email: ochkheti@ifaran.ru}  }
\date{}
\maketitle
\begin{abstract}
This paper examines the Kolmogorov type laws of conserved quantities in   the electron and Hall magnetohydrodynamic  fluids. Inspired by Eyink's longitudinal   structure functions
and recent progress  in classical MHD equations, we derive  four-fifths laws for energy,   magnetic helicity  and generalized helicity in these systems.
  \end{abstract}
\section{Introduction}
\label{intro}
\setcounter{section}{1}\setcounter{equation}{0}
Four-fifths   and four-thirds laws in terms of  third-order structure functions in physical space
are well-known few accurate results  in turbulence. They provide an exact measure of the dissipation rate of energy and  go back to the pioneering works by  Kolmogorov in \cite{[Kolmogorov]} and Yaglom in \cite{[Yaglom]}, respectively.  They can be written as
 \be\label{4543}
\langle[\delta \bm{u}_{L}(\bm{r})]^{3} \rangle=-\f45\epsilon \bm{r}, ~~~~ \langle \delta \bm{u}_{L}(\bm{r})[\delta \theta (\bm{r})]^{2}  \rangle=-\f43\epsilon  \bm{r},
 \ee
 where $\epsilon$ is the mean rate of
  kinetic energy dissipation per unit mass of the Navier-Stokes equations with sufficiently high Reynolds numbers and of the  temperature equation.       Here, $\delta \bm{u}_{L}(\bm{r})=\delta \bm{u} (\bm{r})\cdot\f{\bm{r}}{|\bm{r}|}=[\bm{u}(x+\bm{r})-\bm{u}(x)]\cdot\f{\bm{r}}{|\bm{r}|}$ stands for the longitudinal velocity increment and $\langle\cdot\rangle$ denotes the mean value. It is worth remarking that the  temperature $\theta$ can be replaced by the velocity $\bm{u}$ in \eqref{4543}, which can be found in \cite{[MY]}.   The deduction of \eqref{4543} is intimately connected to the  K\'arm\'arth-Howarth equations.   Without assumptions of homogeneity and isotropy,
Kolmogorov's 4/5 law and Yaglom's 4/3 law were reproduced by  Duchon and Robert \cite{[DR]}, and Eyink \cite{[Eyink1]}, respectively,
  \begin{align}
 &S(\bm{v})=-\f43D(\bm{v}), ~~  S_{L}(\bm{v})=-\f45D(\bm{v}),\label{DRE}
\end{align}
where
\begin{align}
\nonumber&S(\bm{v})=\lim\limits_{\lambda\rightarrow0}S(\bm{v},\lambda)=
\lim\limits_{\lambda\rightarrow0}\f{1}{\lambda}\int_{\partial B }\bm{\ell}\cdot\delta \bm{v}(\lambda\bm{\ell})[\delta \bm{v}(\lambda\bm{\ell})]^{2}\f{d\sigma(\bm{\ell}) }{4\pi}, \\
&D(\bm{v})=\lim\limits_{\varepsilon\rightarrow0}D(\bm{v};\varepsilon)=
 \lim\limits_{\varepsilon\rightarrow0}\f14
\int_{\mathbb{T}^{3}}\nabla\varphi_{\varepsilon}(\bm{\bm{\ell}})\cdot\delta \bm{v}(\bm{\ell})[\delta \bm{v}(\bm{\ell})]^{2}d\bm{\ell},\label{drKHMr}\nonumber\\
&S_{L}(\bm{v})=\lim\limits_{\lambda\rightarrow0}S_{L}(\bm{v},\lambda)=
\lim\limits_{\lambda\rightarrow0}\f{1}{\lambda}\int_{\partial B } \bl \cdot\delta \bm{v}(\lambda\bl) [\delta \bm{v}_{L}(\lambda\bl)]^{2}\f{d\sigma(\bl) }{4\pi}.\nonumber
\end{align} Here, $\sigma(\bm{x})$ represents the surface measure on the sphere $\partial B=\{\bm{x}\in \mathbb{R}^{3}: |\bm{x}|=1\}$ and $\varphi$ is  some smooth non-negative function  supported in $\mathbb{T}^{3}$ with unit integral and $\varphi_{\varepsilon}(\bm{x})=\varepsilon^{-3}\varphi(\f{\bm{x}}{\varepsilon})$.
 The   inertial
anomalous dissipation term    $D(\bm{v};\varepsilon)$ in \eqref{DRE}  originating from the  possible  singularity of rough (weak)
solutions   of the Euler equations prevents the energy conservation.
For more application of this kind dissipation term, the reader can be found in \cite{[Dubrulle],[DE],[KAM],[KFD]}.

 There have been extensive study and application of these laws (see \cite{[Drivas],[DGSH],
[Galtier0],[Galtier],[GPP],[HVLFM],[KFD],[KAM],[MS],[MY],[WWY1],[WC],[YRS],
[Eyink1],[FGSMA],[Frisch],
[DR],[Dubrulle],[Chkhetiani1],[GTW],[AB],[BCPW],[Lindborg],[Novack],[Chandrasekhar],[AOAZ],[BG],[Chkhetiani],[Chkhetiani2]} ).
For a recent  detailed review of the   scaling laws for the magnetohydrodynamics
for the energy
transfer in plasma turbulence,
we would like to refer the reader to  recent work  \cite{[MS]} by  Marino  and Sorriso-Valvo.
Indeed, the first third-order exact law for the magnetohydrodynamics fluid is due to  Politano and
Pouquet \cite{[PP1]}, where they showed that
\be\ba\label{pp1}
&\langle \delta \bm{u}_{L}  (\delta \bm{u} )^{2}  \rangle+  \langle \delta \bm{u}_{L} (\delta \bm{b}  )^{ 2}  \rangle  -2\langle \delta \bm{b}_{L} (\delta \bm{b}  \cdot\delta \bm{u}  )
\rangle =-\f43\epsilon_{E}\bm{r},\\
&2\langle \delta \bm{u}_{L} (\delta \bm{b}  \cdot\delta \bm{u}  )
\rangle- \langle \delta \bm{b}_{L} (\delta \bm{b}  )^{2}   \rangle-\langle \delta \bm{b}_{L} (\delta \bm{u} )^{2} \rangle =-\f43\epsilon_{C} \bm{r}.
\ea\ee
where $\epsilon_{E}$  and $\epsilon_{C}$ stands for
  the mean dissipation rates of total energy and cross-helicity. Simultaneously, Politano and
Pouquet \cite{[PP2]} also deduced the following relations
 \be\label{pp2}\ba
&\langle[\delta \bm{u}_{L}({\bm r})]^{3} \rangle- 6\langle {\bm b}^{2}_{L}\delta \bm{u}_{L}({\bm r})  \rangle=-\f45\epsilon_{E} \bm{r},\\
&\langle[\delta {\bm b }_{L}({\bm r})]^{3} \rangle- 6\langle {\bm b}^{2}_{L}\delta \bm{u}_{L}({\bm r})  \rangle=-\f45\epsilon_{C} \bm{r},
\ea
 \ee
 (see also \cite{[Chandrasekhar]}). Here, $\bm{u}$ and $\bm{b}$ describe the flow velocity field and the magnetic field, respectively.
 Recently, making full use of Eyink's longitudinal structure functions in \cite{[Eyink1]} below
 \be\label{eyinksp}
\ba
&  \bm{u}_{L}(\bx, t,\bl)=(\bn\otimes \bn)  \bm{u}(\bx,t,\bl),~~~~\delta \bm{u}_{T}(\bx, t,\bl)=({\bf 1}-\bn\otimes \bn)  \bm{u}(\bx,t,\bl),
\ea\ee
and the
 it was shown   in \cite{[YWC]} that, for magnetized fluids, there hold
 \be\ba
 S_{EL}(\bm{u}, \bm{b})=-\f45 D_{E}(\bm{u},\bm{b}),~~
 S_{C  L}(\bm{u}, \bm{b})=-\f45 D_{CH}(\bm{u},\bm{b}),
\ea\ee
which correspond  to
\be\ba\label{wyc}
&\langle \delta \bm{u}_{L}  [\delta \bm{u}_{L}|^{2}  \rangle+  \langle \delta \bm{u}_{L} [\delta \bm{b}_{L} |^{ 2}  \rangle  -2\langle \delta \bm{b}_{L} (\delta \bm{b}_{L} \cdot\delta \bm{u}_{L} )
\rangle-\f45\langle \delta \bm{b}_{L}(\delta \bm{b} \cdot\delta \bm{v} )  \rangle+\f45\langle \delta \bm{v}_{L}[\delta \bm{b} | ^{2}  \rangle=-\f45\epsilon_{E} \bm{r},\\
&2\langle \delta \bm{u}_{L} (\delta \bm{b}_{L} \cdot\delta \bm{u}_{L} )
\rangle- \langle \delta \bm{b}_{L} [\delta \bm{b}_{L} |^{2}   \rangle-\langle \delta \bm{b}_{L} ([\delta \bm{u}_{L}|^{2} )-\f45\langle \delta \bm{b}_{L} ([\delta \bm{u} |^{2} )
\rangle+\f45\langle \delta \bm{v}_{L} ( \delta \bm{u}\cdot \delta \bm{b}  )
\rangle=-\f45\epsilon_{C} \bm{r}.
\ea\ee
The analysis of  interaction of different physical quantities plays an important role in   the derivation of Kolmogorov type law in \cite{[YWC]}.
Note that all the above exact scaling relations   \eqref{pp1}-\eqref{wyc}   rely on the traditional MHD equations.
The standard  MHD approximation describes the macroscopic evolution of an electrically conducting single fluid   and
is not satisfied for
   description of small-scale magnetized plasmas.
Actually,  the   electron (EMHD) and Hall (Hall-MHD) magnetohydrodynamic   systems are more effective  than the standard MHD model at the ion inertial length scale, where the motion of the ions can be neglected and the electrons remain quasi-neutrality. The Hall term is predominant in this range. As a result, the EMHD and Hall-MHD equations are widely used in   solar wind,   crust of neutron stars, plasma solids, tokamaks, and  magnetic reconnection(see \cite{[ZYZ],
[Vainshtein73],[Lyut13],[Biskamp99],[Chen16],[Cho11],[ChoLaz],
[Galtier1],[Galtier2],[GP92],[GKR],[KCY],[KP22],[PFLVMH]}). The exact relations of invariant  quantities in EMHD equations and Hall-MHD equations have attracted considerable attention       in the past two decades (see \cite{[BG],[Chkhetiani2],[DGSH],[FGSMA],[Galtier0],[Galtier],[HVLFM],
[WC]}  and references therein). In the case of the EMHD equations, the 4/5 law and 5/12 law of
helicity  were discovered by Chkhetiani in \cite{[Chkhetiani]}. Regarding the case of the Hall-MHD equations,
Galtier deduced
the von K\'arm\'an-Howarth equations for the Hall-MHD flows  and   exact
scaling laws for the third-order correlation tensors   in \cite{[Galtier0]}.
Hellinger,  Verdini,  Landi, Franci  and  Matteini  \cite{[HVLFM]} derived
the Yaglom's 4/3  law of energy  (see also \cite{[FGSMA],[WC]}) in the Hall-MHD equations. Besides,
the  four-thirds law      of  magnetic helicity in EMHD and Hall-MHD systems can be found in \cite{[WC]}.
However, to the best knowledge of authors, limited work has been done in four-fifth    laws of energy in the EMHD and Hall-MHD equations. The objective of this current paper is to consider this issue. For the convenience of presentation,  let $\bm{J}=\nabla\times \bm{b}$ denote the electric current. In what follows, we begin with invocation of  Eyink's longitudinal structure functions \eqref{eyinksp} and study the exact relations of conserved quantities in  electron and Hall magnetohydrodynamic
fluids. The first 4/5 law of energy in the EMHD  equations can be formulated as follows.
\begin{theorem}\label{the1.1}
 Suppose that   $\bm{b}$ satisfy
  the following inviscid EMHD  equations
\be\label{EMHD} \partial_{t}\bm{b} +
 \text{d}_{\text{I}}\nabla\times[(\nabla\times \bm{b})\times \bm{b} ]  =0.
 \ee
where $\text{d}_{\text{I}}$ is the ion inertial length.
Then, there hold the following local longitudinal and transverse K\'arm\'arth-Howarth equations for energy, respectively
\be\ba\label{1.12}
 & \partial_{t}(\bm{b}_{L}^{\varepsilon}\cdot\bm{b})+\f{\text{d}_{\text{I}}}{2}\s\B\{[\s(\bm{b} \otimes \bm{b}_{L})]^{\varepsilon}\times \bm{b}\B\}+\f{\text{d}_{\text{I}}}{2}\s\B\{[\s(\bm{b} \otimes \bm{b})]\times \bm{b}_{L}^{\varepsilon}\B\}+\f{\text{d}_{\text{I}}}{2}\s[\bm{b} (\bm{b}\cdot \bm{J}_{L}^{\varepsilon})] \\&+\f{\text{d}_{\text{I}}}{2}\s[\bm{b} (\bm{b}_{L}^{\varepsilon}\cdot \bm{J})]- \f{\text{d}_{\text{I}}}{2}\s[\bm{J} (\bm{b}\cdot \bm{b}_{L}^{\varepsilon})]
-
\f{\text{d}_{\text{I}}}{4}
\s \B[\bm{J}(\bm{b}_L\cdot\bm{b}_L)\B]^\varep +\f{\text{d}_{\text{I}}}{4}\s\B[\bm{J}(\bm{b}_L\cdot\bm{b}_L)^\varep\B] \\&+\f{\text{d}_{\text{I}}}{2}\s\B[\bm{b} (\bm{J}_L\cdot \bm{b}_L)\B]^\varep-\f{\text{d}_{\text{I}}}{2}\s \B[\bm{b}(\bm{J}_L\cdot \bm{b}_L)^\varep\B]
 =
 -\f{2}{3}D_{EL}^{\varepsilon}(\bm{b},\bm{J}),
\ea\ee
and
\be\ba\label{1.13}
 & \partial_{t}(\bm{b}_{T}^{\varepsilon}\cdot\bm{b})
 +\f{\text{d}_{\text{I}}}{2}\s\B\{[\s(\bm{b} \otimes \bm{b}_{T})]^{\varepsilon}\times \bm{b}\B\}+\f{\text{d}_{\text{I}}}{2}\s\B\{[\s(\bm{b} \otimes \bm{b})]\times \bm{b}_{T}^{\varepsilon}\B\}+\f{\text{d}_{\text{I}}}{2}\s[\bm{b} (\bm{b}\cdot \bm{J}_{T}^{\varepsilon})] \\&+\f{\text{d}_{\text{I}}}{2}\s[\bm{b} (\bm{b}_{T}^{\varepsilon}\cdot \bm{J})]- \f{\text{d}_{\text{I}}}{2}\s[\bm{J} (\bm{b}\cdot \bm{b}_{T}^{\varepsilon})]
 - \f{\text{d}_{\text{I}}}{4}\s \B[\bm{J}(\bm{b}_T\cdot\bm{b}_T)\B]^\varep+\f{\text{d}_{\text{I}}}{4}\s\B[\bm{J}(\bm{b}_T\cdot\bm{b}_T)^\varep\B]
\\&+\f{\text{d}_{\text{I}}}{2} \s\B[\bm{b} (\bm{J}_T\cdot \bm{b}_T)\B]^\varep-\f{\text{d}_{\text{I}}}{2} \s [\bm{b}(\bm{J}_T\cdot \bm{b}_T)^\varep]
= -\f43D_{ET}^{\varepsilon}(\bm{b},\bm{J}).
\ea\ee
where the dissipation terms (K\'arm\'an-Howarth-Monin type relation) are defined by
$$\ba
 D_{EL}^{\varepsilon}(\bm{b},\bm{J})
=&\f{3\text{d}_{\text{I}}}{4 }\int_{\mathbb{T}^{3}} \nabla \varphi^\varep(\ell)\cdot \delta \bm{b}(\bl)(\delta \bm{J}_L \cdot \delta \bm{b}_L)\\&+\f{2}{|\bl|}\varphi^\varep(\ell)\B[\bn\cdot \delta \bm{b}(\bl)(\delta \bm{J}_T\cdot \bm{b}_T)-\delta \bm{b}(\delta \bm{J} \cdot \delta \bm{b})+\delta \bm{J} (\delta \bm{b}\cdot \delta \bm{b})  \B]  d^3 \bl\\
&-\f{3\text{d}_{\text{I}}}{8 }\int_{\mathbb{T}^{3}} \nabla \varphi^\varep(\ell)\cdot \delta \bm{J} (\bl)[\delta \bm{b}_L(\bl)]^2+\f{2}{|\bl|}\varphi^\varep(\ell)\bn\cdot  \delta \bm{J} (\bl)[\delta \bm{b}_T(\bl)]^2  d^3\bl.
\ea$$
and
$$\ba
 D_{ET}^{\varepsilon}(\bm{b},\bm{J})
 =&-\f{3\text{d}_{\text{I}}}{16}\int_{\mathbb{T}^{3}} \nabla \varphi^\varep(\ell)\cdot \delta \bm{J}(\bl)[\delta \bm{b}_T]^2-\f{2}{|\bl|}\varphi^\varep\bn\cdot \delta \bm{J} [\delta \bm{b}_T]^2  d^3\bl\\  &+\f{3\text{d}_{\text{I}}}{8}      \int_{\mathbb{T}^{3}} \nabla \varphi^\varep(\ell)\cdot \delta \bm{b}(\bl)(\delta \bm{J}_T\cdot \delta \bm{b}_T)-\f{2}{|\bl|}\varphi^\varep\bn\cdot \delta\bm{b} (\delta\bm{J}_T\cdot \delta\bm{b}_T)d^3\bl\\&+\f{3\text{d}_{\text{I}}}{8}      \int_{\mathbb{T}^{3}} \f{2}{|\bl|}\bn\cdot\B[\delta \bm{b}(\delta\bm{b}\cdot \delta\bm{J})-\delta \bm{J}(\delta \bm{b}\cdot \delta \bm{b})\B] d^3\bl.
\ea$$	
	In addition, if suppose that for any $1<m,n<\infty, 3\leq p,q<\infty$ with $\f{2}{p}+\f{1}{m}=1$ and $\f{2}{q}+\f{1}{n}=1$ such that $(\bm{b}, \bm{J})$ satisfies
	\be\label{a8}
	\bm{b}\in L^p(0,T;L^q(\mathbb{T}^3))~\text{and}~\bm{J}\in L^m(0,T;L^n(\mathbb{T}^3)).
	\ee	
	Then the function $D_{EX}^\varep(\bm{b},\bm{j})$ with $X=L~\text{or}~T$ converges to a distribution $D_{E}(\bm{b},\bm{J})$ in the sense of distributions as $\varep\to 0$, and $D_{E}(\bm{b},\bm{J})$ satisfies the local equation of energy
\be\label{1.11}
 \partial_{t}(\f12|\bm{b}|^{2})  +\f{\text{d}_{\text{I}}}{2}\s( [\s(\bm{b} \otimes \bm{b})\times \bm{b}])
-\f{\text{d}_{\text{I}}}{4}\s (\bm{J}|\bm{b}|^{2})+\f{\text{d}_{\text{I}}}{2}\s[\bm{b}(\bm{J}\cdot \bm{b})] =D_{E}(\bm{b},\bm{J}),
\ee
Furthermore, there hold the following 4/5 law for the energy
\be\label{45helicity}
S_{EL} (\bm{b},\bm{J})= -\f45 D_{E}(\bm{b},\bm{J}),
\ee
and 8/15 law
$$S_{ET}(\bm{b},\bm{J})= -\f{8}{15} D_{E}(\bm{b},\bm{J}),
$$
where $$S_{EX}(\bm{b},\bm{J})= \lim\limits_{\lambda\to 0}S_{EX}(\bm{b},\bm{J},\lambda),~\text{with}~X=L,T,$$
and
$$\ba
S_{EL}(\bm{b},\bm{J},\lambda)
=&
 \f{1}{\lambda}\int_{\partial B }  \bn \cdot \B[ \delta \bm{b}(\lambda\bl)[\delta \bm{J}_L(\lambda\bl)\cdot \delta \bm{b}_L(\lambda\bl)]-\f12\delta \bm{J}(\lambda\bl)[\delta \bm{b}_T(\lambda\bl)]^2 \B]\f{d\sigma(\bl) }{4\pi}\\&-\f25 \f{1}{\lambda}\int_{\partial B } \bn \cdot\B[\delta\bm{J} (\lambda\bl) [\delta\bm{b}(\lambda\bl) |^{2}-\delta\bm{b}(\lambda\bl)   (\delta\bm{J}(\lambda\bl) \cdot \delta\bm{b}(\lambda\bl) ) \B] \f{d\sigma(\bl) }{4\pi},
\\  S_{ET}(\bm{b},\bm{J},\lambda)
=&
 \f{1}{\lambda}\int_{\partial B } \bn \cdot \B[ \delta \bm{b}(\lambda\bl)[\delta \bm{J}_T(\lambda\bl)\cdot \delta \bm{b}_T(\lambda\bl)]-\f12\delta \bm{J}(\lambda\bl)[\delta \bm{b}_T(\lambda\bl)]^2\B] \f{d\sigma(\bl) }{4\pi}\\&+\f25 \f{1}{\lambda}\int_{\partial B } \bn \cdot \B[\delta\bm{J} (\lambda\bl) [\delta\bm{b}(\lambda\bl) |^{2}-\delta\bm{b}(\lambda\bl)   [\delta\bm{J}(\lambda\bl) \cdot \delta\bm{b}(\lambda\bl) ]\B] \f{d\sigma(\bl) }{4\pi}.
\ea$$
\end{theorem}
\begin{remark} This theorem extends Kolmogorov type 4/5 law of energy from hydrodynamic fluids to electron magnetohydrodynamics turbulence.
\end{remark}
\begin{remark}The four-fifths law  \eqref{45helicity}
 corresponds to
$$\langle\delta\bm{b}_{L} (\delta\bm{J}_{L}\cdot \delta\bm{b}_{L})\rangle-\f12\langle\delta\bm{J}_{L} (\delta\bm{b}_{L})^{2}\rangle-\f25\langle\delta\bm{J}_{L} (\delta\bm{b} )^{2}\rangle+\f25\langle\delta\bm{b}_{L}  (\delta\bm{J} \cdot \delta\bm{b} )\rangle=-\f45\epsilon_{E} r.$$
\end{remark}\begin{remark}According  to vector triple product formula, we can reformulate
$S_{EL}(\bm{b},\bm{J},\lambda)$ and $S_{EL}(\bm{b},\bm{J},\lambda)$ in terms of $\delta\bm{b}\times(\delta\bm{J}\times\delta\bm{b})$.
\end{remark}
  The nonlinear term of the EMHD equations \eqref{EMHD} is a  Hall term involving
second order derivatives rather than the convection term based on
one order derivatives in the Euler and the standard MHD equations in \cite{[YWC]}. Hence, it seems that the analysis of the Hall term is much more difficult. To this end, we shall use the different   equivalent forms of the Hall term and the identity \eqref{2.9} observed in \cite{[YWC]} to deal with the interaction between the   magnetic field and  the electric current.   As a byproduct,
this together with the recent four-fifths law of total  energy in the conventional MHD equations in \cite{[YWC]} leads to the scaling exact law of total energy in the Hall-MHD equations. The verification   is left to the reader.
\begin{coro}
 Suppose that the   triplet $(\bm{v},\bm{b}, \Pi)$  satisfy the following hall-MHD equations
\be\label{hallMHD}\left\{\ba
&\partial_{t} \bm{u}  +\s (\bm{u}\otimes \bm{u} )  -\s (\bm{b}\otimes \bm{b} )  + \nabla P =0, \\
&\partial_{t} \bm{b} +\s (\bm{u}\otimes \bm{b} )   -\s (\bm{b}\otimes \bm{ u} ) +
 \text{d}_{\text{I}}\nabla\times[(\nabla\times \bm{b})\times \bm{b} ] =0, \\
&\Div \bm{u} =\Div \bm{b} =0,
 \ea\right.\ee
where $ \bm{u}$, $\bm{b}$ and $\Pi$ stand for the veolcity, magnetic field and the pressure of the fluid, respectively. Then, there hold the following local longitudinal and transverse K\'arm\'arth-Howarth equations for energy,
  $$\ba
 &\partial_{t}( \bm{u}_{L}^{\varepsilon}\cdot\bm{u} +\bm{b}\cdot\bm{b}_{L}^{\varepsilon})+\s\B[\bm{u}( \bm{u}\cdot \bm{u}_{L }^{\varepsilon})-\bm{b}(\bm{b} \cdot \bm{u}_{L }^{\varepsilon})
 +\bm{u}(\bm{b} \cdot \bm{b}_{L }^{\varepsilon})
 -\bm{b}(\bm{u} \cdot \bm{b}_{L }^{\varepsilon})\B]
 \\&+\s\B[P_L^\varep \bm{u}+P \bm{u}_L^\varep\B]
 -\s\B[\big(\bm{b} (\bm{u}_L\cdot \bm{b}_L)\big)^\varep-\bm{b}(\bm{u}_L\cdot \bm{b}_L)^\varep\B]\\& +\f12\s \B[\big(\bm{u}(\bm{u}_L\cdot\bm{u}_L)\big)^\varep-\bm{u}(\bm{u}_L\cdot\bm{u}_L)^\varep\B]+\f12
 \s \B[\big(\bm{u}(\bm{b}_L\cdot\bm{b}_L)\big)^\varep
 -\bm{u}(\bm{b}_L\cdot\bm{b}_L)^\varep\B]
 \\&+\f{\text{d}_{\text{I}}}{2}\s\B\{[\s(\bm{b} \otimes \bm{b}_{L})]^{\varepsilon}\times \bm{b}\B\}+\f{\text{d}_{\text{I}}}{2}\s\B\{[\s(\bm{b} \otimes \bm{b})]\times \bm{b}_{L}^{\varepsilon}\B\}+\f{\text{d}_{\text{I}}}{2}\s[\bm{b} (\bm{b}\cdot \bm{J}_{L}^{\varepsilon})] \\&+\f{\text{d}_{\text{I}}}{2}\s[\bm{b} (\bm{b}_{L}^{\varepsilon}\cdot \bm{J})]- \f{\text{d}_{\text{I}}}{2}\s[\bm{J} (\bm{b}\cdot \bm{b}_{L}^{\varepsilon})]
 -
 \f{\text{d}_{\text{I}}}{4}
 \s \B[\bm{J}(\bm{b}_L\cdot\bm{b}_L)\B]^\varep +\f{\text{d}_{\text{I}}}{4}\s\B[\bm{J}(\bm{b}_L\cdot\bm{b}_L)^\varep\B] \\&+\f{\text{d}_{\text{I}}}{2}\s\B[\bm{b} (\bm{J}_L\cdot \bm{b}_L)\B]^\varep-\f{\text{d}_{\text{I}}}{2}\s \B[\bm{b}(\bm{J}_L\cdot \bm{b}_L)^\varep\B] =  -\f{2}{3}D^{\varepsilon}_{EL}(\bm{b},\bm{J});
\ea$$
and
$$\ba &\partial_{t}( \bm{u}_{T}^{\varepsilon}\cdot\bm{u} +\bm{b}\cdot\bm{b}_{T}^{\varepsilon})+\s\B[\bm{u}( \bm{u}\cdot \bm{u}_{T }^{\varepsilon})
-\bm{b}(\bm{b} \cdot \bm{u}_{T }^{\varepsilon})
+\bm{u}(\bm{b} \cdot \bm{b}_{T }^{\varepsilon})
-\bm{b}(\bm{u} \cdot \bm{b}_{T }^{\varepsilon})\B]
\\&+\s\B[P_T^\varep \bm{u}+P \bm{u}_T^\varep\B]
-\text{div} \B[\big(\bm{b}(\bm{u}_T\cdot \bm{b}_T)\big)^\varep-\bm{h}(\bm{u}_T\cdot \bm{b}_T)^\varep\B]\\& +\f12\s \B[\big(\bm{u}(\bm{u}_T\cdot\bm{u}_T)\big)^\varep-\bm{u}(\bm{u}_T\cdot\bm{u}_T)^\varep\B]+\f12\s \B[\big(\bm{u}(\bm{b}_T\cdot\bm{b}_T)\big)^\varep
-\bm{u}(\bm{b}_T\cdot\bm{b}_T)^\varep\B] \\&+\f{\text{d}_{\text{I}}}{2}\s\B\{[\s(\bm{b} \otimes \bm{b}_{T})]^{\varepsilon}\times \bm{b}\B\}+\f{\text{d}_{\text{I}}}{2}\s\B\{[\s(\bm{b} \otimes \bm{b})]\times \bm{b}_{T}^{\varepsilon}\B\}+\f{\text{d}_{\text{I}}}{2}\s[\bm{b} (\bm{b}\cdot \bm{J}_{T}^{\varepsilon})] \\&+\f{\text{d}_{\text{I}}}{2}\s[\bm{b} (\bm{b}_{T}^{\varepsilon}\cdot \bm{J})]- \f{\text{d}_{\text{I}}}{2}\s[\bm{J} (\bm{b}\cdot \bm{b}_{T}^{\varepsilon})]
- \f{\text{d}_{\text{I}}}{4}\s \B[\bm{J}(\bm{b}_T\cdot\bm{b}_T)\B]^\varep+\f{\text{d}_{\text{I}}}{4}\s\B[\bm{J}(\bm{b}_T\cdot\bm{b}_T)^\varep\B]
\\&+\f{\text{d}_{\text{I}}}{2} \s\B[\bm{b} (\bm{J}_T\cdot \bm{b}_T)\B]^\varep-\f{\text{d}_{\text{I}}}{2} \s [\bm{b}(\bm{J}_T\cdot \bm{b}_T)^\varep] =-\f43D_{ET}^{\varepsilon}(\bm{b},\bm{J});
   \ea$$
 In addition, if suppose that   $(\bm{u}, \bm{b})$ satisfies
	\be\label{1.19}
	\bm{u},\bm{b}\in L^3(0,T;L^3(\mathbb{T}^3)),\bm{b}\in L^p(0,T;L^q(\mathbb{T}^3))~\text{and}~\bm{J}\in L^m(0,T;L^n(\mathbb{T}^3)),
	\ee	where
$1<m,n<\infty, 3\leq p,q<\infty$ with $\f{2}{p}+\f{1}{m}=1$ and $\f{2}{q}+\f{1}{n}=1$.
Then the function $D_{EX}^\varep(\bm{u},\bm{b})$ with $X=L,T$   converges to a distribution $D_{E}(\bm{u},\bm{b})$ in the sense of distributions as $\varep\to 0$, and $D_{E}(\bm{u},\bm{b})$ satisfies the local equation of total energy
\be \ba\label{1.18}
&\partial_{t}(\f{|\bm{u}|^{2}+|\bm{b}|^{2}}{2}  )  +\s\B[\bm{u}\B( \f12(|\bm{u}|^{2}+|\bm{b} |^{2})+\Pi\B)-\bm{b}(\bm{b}\cdot \bm{u})\B] \\& +\f{\text{d}_{\text{I}}}{2}\s( [\s(\bm{b} \otimes \bm{b})\times \bm{b}]
-\f{\text{d}_{\text{I}}}{4}\s (\bm{J}|\bm{b}|^{2})+ \f{\text{d}_{\text{I}}}{2}\s[\bm{b}(\bm{J}\cdot \bm{b})] =-D_{E}(\bm{u},\bm{b},\bm{J}),
 \ea\ee

Furthermore, we have the following 4/5 law and 8/15 laws
\begin{align} \label{mhde45}S_{EL}(\bm{u}, \bm{b})=-\f45 D_{E}(\bm{u},\bm{b},\bm{J}),~ S_{ET}(\bm{u}, \bm{b})=-\f{8}{15}D_{E}(\bm{u},\bm{b},\bm{J}),
\end{align}
where $$S_{EX}(\bm{u},\bm{b})= \lim\limits_{\lambda\to 0}S_{EX}(\bm{u},\bm{b},\lambda),~\text{with}~X=L,T,  $$
and
 $$\ba
& S_{EL}(\bm{u}, \bm{b},\lambda)\\
  = &\f{1}{\lambda}\int_{\partial B } \bn \cdot \B[ \delta \bm{v}(\lambda\bl)|\B(\delta \bm{v}_L(\lambda\bl)]^2+[\delta \bm{b}_L(\lambda\bl)]^2\B)-2\delta \bm{b}(\bl)\B(\delta \bm{v}_L(\lambda\bl) \cdot \delta \bm{b}_L(\lambda\bl)\B)\B] \f{d\sigma(\bl) }{4\pi}\\  &+\f45 \f{1}{\lambda}\int_{\partial B } \bn \cdot \B[ \delta \bm{b}(\lambda\bl) \B(\delta\bm{b}(\lambda\bl)\cdot\delta \bm{v}(\lambda\bl)\B)-\delta \bm{v}(\lambda\bl) | \delta\bm{b}(\lambda\bl)]^{2}\B] \f{d\sigma(\bl) }{4\pi}\\&+ \f{1}{\lambda}\int_{\partial B }  \bl \cdot \B[ \delta \bm{b}(\lambda\bl)(\delta \bm{J}_L(\lambda\bl)\cdot \delta \bm{b}_L(\lambda\bl))-\f12\delta \bm{J}(\lambda\bl)[\delta \bm{b}_T(\lambda\bl)]^2 \B]\f{d\sigma(\bl) }{4\pi}\\&-\f25 \f{1}{\lambda}\int_{\partial B } \bl \cdot\B[\delta\bm{J} (\lambda\bl) [\delta\bm{b} (\lambda\bl)]^{2}-\delta\bm{b}  (\lambda\bl) (\delta\bm{J} \cdot \delta\bm{b} (\lambda\bl)) \B] \f{d\sigma(\bl) }{4\pi};\\
 & S_{ET}(\bm{u}, \bm{b},\lambda)\\
 =& \f{1}{\lambda}\int_{\partial B } \bn \cdot \B[\delta \bm{v}(\lambda\bl)\B( [\delta \bm{v}_T(\lambda\bl)]^2+ [\delta \bm{b}_T(\lambda \bl)]^2\B)-2 \delta\bm{b}(\lambda\bl) \B(\delta\bm{v}_T(\lambda\bl)\cdot \delta\bm{b}_T(\lambda\bl)\B)\B] \f{d\sigma(\bl) }{4\pi}\\~~~~~~~~~\,&  -\f45 \f{1}{\lambda}\int_{\partial B } \bn\cdot \B[\delta \bm{b}(\lambda\bl) \B(\delta\bm{b}(\lambda\bl)\cdot\delta \bm{v}(\lambda\bl)\B)-\delta \bm{v}(\lambda\bl) | \delta\bm{b}(\lambda\bl)]^{2}\B] \f{d\sigma(\bl) }{4\pi}\\&+ \f{1}{\lambda}\int_{\partial B } \bl \cdot \B[ \delta \bm{b}(\lambda\bl)(\delta \bm{J}_T(\lambda\bl)\cdot \delta \bm{b}_T)-\f12\delta \bm{J}(\lambda\bl)[\delta \bm{b}_T]^2(\lambda\bl)\B] \f{d\sigma(\bl) }{4\pi}\\&+\f25 \f{1}{\lambda}\int_{\partial B } \bl \cdot \B[\delta\bm{J} (\lambda\bl) [\delta\bm{b} (\lambda\bl)]^{2}-\delta\bm{b}  (\lambda\bl) (\delta\bm{J}(\lambda\bl) \cdot \delta\bm{b}(\lambda\bl) )\B] \f{d\sigma(\bl) }{4\pi}.
\ea$$
  \end{coro}

  \begin{remark} The four-fifths law in \eqref{mhde45} can be written as
  $$\ba
 &\langle \delta \bm{u}_{L}  [\delta \bm{u}_{L}|^{2}  \rangle+  \langle \delta \bm{u}_{L} [\delta \bm{h}_{L} |^{ 2}  \rangle  -2\langle \delta \bm{h}_{L} (\delta \bm{h}_{L} \cdot\delta \bm{u}_{L} )
\rangle-\f45\langle \delta \bm{h}_{L}(\delta \bm{h} \cdot\delta \bm{v} )  \rangle+\f45\langle \delta \bm{v}_{L}[\delta \bm{h} | ^{2}  \rangle\\&+\langle\delta\bm{b}_{L} (\delta\bm{J}_{L}\cdot \delta\bm{b}_{L})\rangle-\f12\langle\delta\bm{J}_{L} [\delta\bm{b}_{L}|^{2}\rangle-\f25\langle\delta\bm{J}_{L} [\delta\bm{b} |^{2}\rangle+\f25\langle\delta\bm{b}_{L}  (\delta\bm{J} \cdot \delta\bm{b} )\rangle=-\f45\epsilon_{E} r.
\ea$$
  \end{remark}
Next, we turn our attention to magnetic helicity $\int_{\mathbb{T}^3}  \bm{ A} \cdot {\rm curl\,}  \bm{ A}\  dx  $ as the second conserved quantity in the EMHD and Hall-MHD equations, where the magnetic vector potential
   $\bm{ A}={\rm curl}^{-1}\bm{b}$ is governed by
    \be\label{hmpotentialeq}
\bm{A}_{t}-\bm{u}\times \bm{b}+\text{d}_{\text{I}} (\nabla\times \bm{b})\times \bm{b}   +\nabla \pi=0, \s\bm{A}=0.
\ee
As mentioned above, the 4/5 law of  magnetic helicity in the EMHD equations had been discovered by  Chkhetiani in \cite{[Chkhetiani2]}. Hence, we mainly focus on four-fifths law  of   the magnetic helicity in the  Hall-MHD system.
\begin{theorem}\label{the1.3}
Assume  that the pair $(\bm{b},\bm{u}) $ and $\bm{ A}$ be the solution of the HMHD equations and   \eqref{hallMHD}  and \eqref{hmpotentialeq}, respectively.
There holds the following K\'arm\'an-Howarth-Monin type equation
$$\ba
&\partial_t(\bm{A}_{L}^{\varepsilon}\cdot \bm{b}+\bm{A}\cdot \bm{b}_{L}^{\varepsilon})-\s[(\bm{u}\times \bm{b}_{L})^{\varepsilon}\times \bm{A}]-2(\bm{u}\times \bm{b}_{L})^{\varepsilon}\cdot \bm{b}-\s[(\bm{u}\times \bm{b})\times \bm{A}_{L}^{\varepsilon}]\\&-2(\bm{u}\times \bm{b})\cdot \bm{b}^{\varepsilon}_{L}+ \text{d}_{\text{I}}\s([\s(\bm{b}\otimes \bm{b}_{L})]^{\varepsilon}\times \bm{A})
+\text{d}_{\text{I}}\s([\s(\bm{b}\otimes \bm{b})]\times \bm{A}_{L}^{\varepsilon})\\
& +\s[ \pi_L^{\varepsilon}\bm{b}+ \pi \bm{b}_{L}^{\varepsilon}]+2\text{d}_{\text{I}}\s\B[\bm{b}(\bm{b}\cdot\bm{b}_L^\varep)  \B]+\text{d}_{\text{I}}\s \B[\big(\bm{b}(\bm{b}_T\cdot\bm{b}_T)\big)^\varep-\bm{b}(\bm{b}_T\cdot\bm{b}_T)^\varep\B]
\\
=&-\f43 D_{ML}^{\varepsilon}(\bm{b},\bm{b}),
  \ea$$
and
$$\ba
&\partial_t(\bm{A}_{T}^{\varepsilon}\cdot \bm{b}+\bm{A}\cdot \bm{b}_{T}^{\varepsilon})-\s[(\bm{u}\times \bm{b}_{T})^{\varepsilon}\times \bm{A}]-2(\bm{u}\times \bm{b}_{T})^{\varepsilon}\cdot \bm{b}-\s[(\bm{u}\times \bm{b})\times \bm{A}_{T}^{\varepsilon}]\\&-2(\bm{u}\times \bm{b})\cdot \bm{b}^{\varepsilon}_{T}+ \text{d}_{\text{I}}\s([\s(\bm{b}\otimes \bm{b}_{T})]^{\varepsilon}\times \bm{A})
+\text{d}_{\text{I}}\s([\s(\bm{b}\otimes \bm{b})]\times \bm{A}_{T}^{\varepsilon})\\
& +\s[ \pi_T^{\varepsilon}\bm{b}+ \pi \bm{b}_{T}^{\varepsilon}]+2\text{d}_{\text{I}}\s\B[\bm{b}(\bm{b}\cdot\bm{b}_T^\varep)  \B]+ \s \B[\big(\bm{b}(\bm{b}_T\cdot\bm{b}_T)\big)^\varep- \bm{b}(\bm{b}_T\cdot\bm{b}_T)^\varep\B]\\
      = &-\f83D_{M T}^{\varepsilon}(\bm{b},\bm{b}),
\ea$$
where the  $D_{MX}^{\varepsilon}(\bm{b},\bm{b})$ for $X=L,T$ is defined in $$\ba
&D_{M L}^{\varepsilon}(\bm{b},\bm{b})=\f34\text{d}_{\text{I}}\int_{\mathbb{T}^{3}} \nabla \varphi^\varep(\ell)\cdot   \delta \bm{b} [\delta \bm{b}_L]^2+\f{2}{|\bl|}\varphi^\varep\bn\cdot   \delta \bm{b} [\delta \bm{b}_T(\bl)]^2 d^3\bl,
\\
&D_{M T}^{\varepsilon}(\bm{b},\bm{b}) =\f38\text{d}_{\text{I}}\int_{\mathbb{T}^{3}}\B[\nabla \varphi^\varep(\ell)-\f{2}{|\bl|}\varphi^\varep\bn\B]\cdot  \delta \bm{b} [\delta \bm{b}_T]^2  d^3\bl.
\ea$$
 Let
 $\bm{b}$ be a weak solution of the inviscid HMDH equations \eqref{hallMHD}
and magnetic vector potential $\bm{A}$  satisfy \eqref{hmpotentialeq}.  Assume that
 \be\label{the1.4c}
\bm{b} \in L^{\infty}(0,T;L^{\f32}(\mathbb{T}^{3}))\cap  L^{3}(0,T;L^{3}(\mathbb{T}^{3})) ~\text{and}\ \bm{u}\in L^{3}(0,T;L^{3}(\mathbb{T}^{3})). \ee
  Then the function
$D_{M L}^{\varepsilon}(\bm{b},\bm{b})$  and    $D_{M T}^{\varepsilon}(\bm{b},\bm{b})$
 converge  to a distribution $D_{MH}(\bm{b})$ in the sense of distributions as $\varepsilon\rightarrow0$, and $D_{MH  }(\bm{b})$ satisfies       the local equation of  magnetic helicity
\be \ba\label{01.20}
& \f12\partial_{t}(\bm{b}\cdot\bm{ A} )+\f12\s[(\bm{b}\times\bm{u}) \times \bm{ A}]
 + \f{\text{d}_{\text{I}}}{2}\s([\s(\bm{b} \otimes \bm{b})]\times \bm{ A})    + \f12\s[ \pi \bm{b}]+ \f12 {\text{d}_{\text{I}}}\s(\bm{b}|\bm{b}|^{2})\\
 =&-D_{M }(\bm{b}, \bm{b})
\ea\ee
in the sense of distributions.
Moreover, there hold the following scaling exact relation
\be\label{HMHDMHYaglom4/3law}
D_{M}(\bm{b},\bm{b})=-\f45 {S}_{ML}(\bm{b},\bm{b}),
D_{M}(\bm{b},\bm{b})=-\f{8}{15} {S}_{MT}(\bm{b},\bm{b}).
\ee
where $$S_{MX}(\bm{b},\bm{b})= \lim\limits_{\lambda\to 0}S_{MX}(\bm{b},\bm{b},\lambda),~\text{with}~X=L,T,   $$
and $$\ba
& {S}_{ML}(\bm{b}, \bm{b},\lambda)=\f{1}{\lambda}\int_{\partial B } \bn \cdot \B[   \delta \bm{b} (\lambda\bl)[\delta \bm{b}_L(\lambda\bl)]^2 \B] \f{d\sigma(\bl) }{4\pi},\\
& {S}_{MT}(\bm{b}, \bm{b},\lambda)=\f{1}{\lambda}\int_{\partial B } \bn \cdot \B[  \delta \bm{b} (\lambda\bl)[\delta \bm{b}_T(\lambda\bl)]^2 \B] \f{d\sigma(\bl) }{4\pi}.
\ea$$
 \end{theorem}
 \begin{remark} One may rewrite  the exact relation   \eqref{HMHDMHYaglom4/3law} in the Hall-MHD equations as
 $$\langle[\delta \bm{b}_{L}(\bm{r})]^{3} \rangle=-\f45\epsilon \bm{r},$$
which is   consistent with the  four-fifths   law of  magnetic helicity obtained in the EMHD system by Chkhetiani in \cite{[Chkhetiani2]}.
 \end{remark}
Total energy, cross-helicity and magnetic helicity are three quadratic invariant  quantities in classical MHD approximation.  However, the Hall term in the Hall-MHD equations  \eqref{hallMHD} destroys the conservation of cross-helicity $\int_{\mathbb{T}^{3}} \bm{u}\cdot \bm{b} dx$. Alternatively, the following generalized helicity
 \be\label{ghh}
 \int_{\mathbb{T}^{3}}(\bm{A}+\bm{u})(\bm{b}+\bm{\omega})dx
 \ee
 is as the third quadratic conserved  quantity in the Hall-MHD system \eqref{hallMHD}. This invariant  quantity was initiated by
   Turner
 in \cite{[Turner]}.
The generalized hybrid helicity \eqref{ghh}  helps us   to establish
 Alfv\'en's theorem in the Hall-MHD system and is close to  the topology of the Hall MHD fluid (see \cite{[Galtier],[Galtier06],[Turner]}).
For the generalized helicity \eqref{ghh},
a slight variant of the proof of above theorem means that
\be\label{1.15}\ba
&\langle \delta \bv_{L} (\delta \bv_{L} \cdot\delta \bomega_{L} )  \rangle- \f12\langle \delta \bomega_{L} (\delta \bv_{L} )^{ 2}  \rangle  +2\langle \delta \bv_{L} (\delta \bm{h}_{L} \cdot\delta \bv_{L} )
\rangle -\langle \delta \bm{h}_{L}  (\delta \bv_{L})^{2} \rangle \\&-\f25\langle  \delta\bomega_{L}(\delta\bv)^{2}\rangle+ \f25\langle  \delta\bv_{L}(\delta\bv\cdot\delta\bomega)\rangle -\f45\langle \delta \bm{h}_{L} ( \delta \bv )^{2}
\rangle+\f45\langle \delta \bm{v}_{L} ( \delta \bv\cdot \delta \bm{h}  )
\rangle=-\f45\epsilon_{H} \bm{r}.
\ea\ee
The proof is left to the reader.

The paper is organized as follows.  In section 2, we will state some notations and key    identities which we will be   frequently used in the whole paper. Section 3  is devoted to the scaling law of energy  for the EMHD equations. In section 4, we consider the four-fifths law of magnetic helicity in the Hall-MHD system.  Concluding remarks are given in
section 5.

\section{Notation and preliminaries}

Firstly, we will give some notatiois we will used in the present paper. In the sequel, for any $p\in [1,\,\infty]$, the notation $L^{p}(0,\,T;X)$ stands for the set of measurable functions $f$ on the interval $(0,\,T)$ with values in $X$ and $\|f\|_{X}$ belonging to $L^{p}(0,\,T)$.
We also let $\varphi(\ell)$ be any $C_0^\infty$ function, nonnegative with unit integral, radially symmetric, and $\varphi^\varep(\ell)=\f{1}{\varep^3}\varphi(\f{\ell}{\varep})$.
Then,  for any function $f\in L^1_{\textrm{loc}}(\mathbb{R}^3)$, its mollified version is defined by
$$
f^\varepsilon(\bm{ x})=\int_{\mathbb{R}^3}\varphi_{\varepsilon}(\bl)f(\bm{x}+\bm{\ell})d\bm{\ell},\ \ \bm{x}\in \mathbb{R}^3.
$$
 Just as \cite{[Eyink1]}, for any vector $\bm{E}(x,t)=(E_{3}(x,t),E_{3}(x,t),E_{3}(x,t))$,
we denote
\be\label{defiofeyink}\ba
\bm{E}^\varep_X(\bx,t)=\int_{\mathbb{T}^{3}} \varphi^\varep(\bl)\bm{E}_X(\bx,t;\bl)d^3\bl,~~~X=L,T,
\ea\ee
where
$$\bm{E}_L(\bx,t,\bl)=(\bn\otimes \bn)\cdot \bm{E}(\bx+\bl,t),~~\bm{E}_T(\bx,t,\bl)=({\bf 1}-\bn\otimes \bn)\cdot \bm{E}(\bx+\bl,t).$$
Furthermore, by a straightforward computation, it is easy   to verify that
\be\label{1/3}
\int_{\mathbb{T}^{3}}
\varphi^\varep(\bl) (\bn\otimes \bn) d^3\bl=\f{\delta_{ij}}{3}.
\ee
Secondly, for the convenience of the reader, we recall some vector identities as follows,
\be\ba\label{VI}
&\nabla(\bm{E}\cdot \bm{F})=\bm{E}\cdot\nabla \bm{F}+\bm{F}\cdot\nabla \bm{E}+\bm{E}\times\text{curl}\bm{F}+\bm{F}\times\text{curl}\bm{E},\\
& \nabla\times(\bm{E}\times \bm{F})=\bm{E} \s \bm{F}-\bm{F} \s  \bm{E}+\bm{F}\cdot\nabla \bm{E}-\bm{E}\cdot\nabla \bm{F},\\
&\bm{E}\cdot(\nabla\times \bm{F})=\s(\bm{F}\times \bm{E})+\bm{F}\cdot(\nabla\times \bm{E}),\\
&\s(\nabla \times\bm{E})=0,~~\nabla \times(\nabla \bm{E})=0,
\ea\ee
which will be frequently used in this paper.

Then combining \eqref{VI} and the divgence-free condition $ \Div \bm{J}=0$ with $\bm{J}=\text{curl} \bm{b}$, we find
 \be\label{identity}\ba
& \bm{b}\cdot\nabla  \bm{b} = \f{1}{2}\nabla | \bm{b}|^{2}+\bm{J}\times \bm{b},\\
&\nabla\times(\bm{J}\times \bm{b})= \bm{b}\cdot\nabla \bm{J}-\bm{J}\cdot\nabla \bm{b},\ea\ee
from which follows that
\begin{align}\label{non1}
 &\nabla\times[(\nabla\times \bm{b})\times \bm{b}]= \nabla\times[\bm{J}\times \bm{b}]=
 \s( \bm{b}\otimes\bm{J})- \s(\bm{J}\otimes  \bm{b}),\\
 \label{non2}&\nabla\times[(\nabla\times \bm{b})\times \bm{b}]= \nabla\times[\bm{J}\times \bm{b}]=\nabla\times[\s(\bm{b} \otimes \bm{b})-\nabla\f12|\bm{b}|^{2}]=\nabla\times[\s(\bm{b} \otimes \bm{b})].
\end{align}
According to the definition \eqref{defiofeyink}, we see that
\begin{align}\label{non3}
 &\nabla\times[(\nabla\times \bm{b})\times \bm{b}_{X}]^{\varepsilon} =
 \s( \bm{b}\otimes\bm{J}_{X})^{\varepsilon}- \s(\bm{J}\otimes  \bm{b}_{X})^{\varepsilon},\\
 \label{non4}&\nabla\times[(\nabla\times \bm{b})\times \bm{b}_{X}]^{\varepsilon} =\nabla\times[\s(\bm{b} \otimes \bm{b}_{X})]^{\varepsilon}.
\end{align}
 Next, we present  a critical identity observed in \cite{[YWC]}, which allows one to deal with the
interaction of different physical quantities.
\begin{lemma}\label{lemma2.1}For any vectors ${\bf E},$ ${\bf F} $ and ${\bf G}$ in three dimension, there holds
\be\label{2.9}\ba
&\B[\f{\partial}{\partial_{\ell_k}}(\bn_i\bn_j)-(\f{\partial \bn_i}{\partial\ell_j}+\f{\partial \bn_j}{\partial\ell_i})\bn_k\B]E_kF_iG_j
= \f{1}{|\bl|}\bn\cdot\B[{\bf G}({\bf E}\cdot {\bf F})+{\bf F}({\bf E}\cdot {\bf G})-2{\bf E}({\bf F}\cdot {\bf G})\B].
\ea\ee
\end{lemma}

\begin{proof}
We check
\be\label{c1}\ba
\partial_{\ell_k}\bn_i&=\f{\partial}{\partial\ell_k}\B(\f{\ell_i}{|\bl|}\B)=\f{\delta^{ik}|\bl|-\f{\ell_i\ell_k}{|\bl|}}{|\bl]^2}=\f{1}{|\bl|}(\delta^{ik}-\bn_i\bn_k),
\ea\ee
from which follows that
\be\label{2.10}\ba
&\B[\f{\partial}{\partial_{\ell_k}}(\bn_i\bn_j)-(\f{\partial \bn_i}{\partial\ell_j}+\f{\partial \bn_j}{\partial\ell_i})\bn_k\B]E_kF_iG_j\\
=&\B[\big(\f{\partial \bn_i}{\partial \ell_k}\bn_j+\f{\partial \bn_j}{\partial \ell_k}\bn_i\big)-\f{2}{|\bl|}(\delta^{ij}-\bn_i\bn_j)\bn_k\B]E_kF_iG_j\\
=&\f{1}{|\bl|}\B[(\delta^{ik}-\bn_i\bn_k)\bn_j+(\delta^{jk}-\bn_j\bn_k)\bn_i-2(\delta^{ij}-\bn_i\bn_j)\bn_k\B]E_kF_iG_j\\
=&\f{1}{|\bl|}\B[(\delta^{ik}\bn_j+\delta^{jk}\bn_i-2\delta^{ij}\bn_k  ]E_{k}F_{i}G_{j}.
\ea\ee
For the convience, we denote the left hand side of above equation by $I$. Then, we will discuss term $I$ into five cases.\\
Case1:   $i=k=j$. It is not difficult to verify that
\be\label{2.11}
I=\f{1}{|\bl|}(\bn_j+\bn_i-2\bn_k)E_kF_iG_j={\bf 0}.
\ee
Case 2:  $i=k$ but $j\neq k$. Some tedious manipulation leads to
\be\label{2.12}\ba
I&=\f{1}{|\bl|}\bn_jE_kF_iG_j\\
&=\f{1}{|\bl|}\B[\bn_1G_1(E_2F_2+E_3F_3)+\bn_2G_2(E_1F_1+E_3F_3)+\bn_3G_3(E_1F_1+E_2F_2)\B]\\
&=\f{1}{|\bl|}\B[\bn_1G_1 ({\bf E}\cdot {\bf F})-\bn_1E_1F_1G_1+\bn_2G_2 ({\bf E}\cdot {\bf F})-\bn_2E_2 F_2G_2+\bn_3G_3({\bf E}\cdot {\bf F})-\bn_3E_3F_3G_3\B]\\
&=\f{1}{|\bl|}\B[\bn\cdot {\bf G} ({\bf E}\cdot {\bf F})-\bn_1E_1F_1G_1-\bn_2E_2 F_2G_2-\bn_3E_3F_3G_3\B].
\ea\ee
Case 3:   $i\neq k$ and $j=i$. It is straightforward to show that
\be\label{2.13}\ba
I=&\f{1}{|\bl|}(-2\bn_k)E_kF_iG_j\\
=&-\f{2}{|\bl|}\B[\bn_1E_1 (F_2G_2+F_3G_3)+\bn_2E_2 (F_1G_1+F_3G_3)+\bn_3E_3 (F_1G_1+B_2G_2)\B]\\
=&-\f{2}{|\bl|}\B[\bn\cdot {\bf E}({\bf F}\cdot {\bf G})-\bn_1E_1F_1G_1-\bn_2E_2 F_2G_2-\bn_3E_3F_3G_3\B].
\ea\ee
Case 4:  $i\neq k$ and $j=k$. A simple calculation yields that
\be\label{2.14}\ba
I=&\f{1}{|\bl|}\bn_iE_kF_iG_j\\
=&\f{1}{|\bl|}\B[\bn_1F_1(E_2G_2+E_3G_3)+\bn_2F_2(E_1G_1+E_3G_3)+\bn_3F_3(E_1G_1+E_2G_2)\B]\\
=&\f{1}{|\bl|}\B[\bn\cdot {\bf F}({\bf E}\cdot {\bf G})-\bn_1E_1F_1G_1-\bn_2E_2 F_2G_2-\bn_3E_3F_3G_3\B].
\ea\ee
Case 5:   $i\neq k$, $j\neq k$ and $j\neq i$. Notice that
$I={\bf 0}.$

Putting the above identities together, we get the desired result \eqref{2.9}.
\end{proof}
With Lemma \ref{lemma2.1} in hand, we can further deduce the following identities, which  will be frequently used in the proof of four-fifths laws in the electron and Hall magnetohydrodynamic
fluids.
\begin{lemma}\label{lemma2.2} For any vector
$\bm{E}=(E_{1}(x),E_{2}(x),E_{3}(x)), $ $ \bm{F}=(F_{1}(x),F_{2}(x),F_{3}(x))$, there hold the following identities
\be\label{2.16}\ba
&\int_{\mathbb{T}^{3}} \nabla \varphi^\varep(\ell)\cdot \delta \bm{E}(\delta \bm{F}_L \cdot \delta \bm{E}_L)+\f{2}{|\bl|}\varphi^\varep\bn\cdot \delta \bm{E}(\delta \bm{F}_T\cdot \bm{E}_T)\\&- \f{1}{|\bl|}\varphi^\varep\bn\cdot\B[\delta \bm{E}(\delta \bm{E} \cdot \delta \bm{F})-\delta \bm{F} (\delta \bm{E}\cdot \delta \bm{E})\B] d^3 \bl +\s\B[\big(\bm{E} (\bm{F}_L\cdot \bm{E}_L)\big)^\varep-\bm{E}(\bm{F}_L\cdot \bm{E}_L)^\varep\B]\\&=E_j\partial_k\B[(E_k F_{L_j})^\varep-(E_k F_{L_j}^\varep)\B]+F_i\partial_k\B[(E_k E_{L_i})^\varep-(E_k E_{L_i}^\varep)\B];
\ea\ee
\be\label{2.17}\ba
&\f12\int_{\mathbb{T}^{3}} \nabla \varphi^\varep(\ell)\cdot \delta \bm{E} [\delta \bm{F}_L]^2+\f{2}{|\bl|}\varphi^\varep\bn\cdot \B[\delta \bm{E} [\delta \bm{F}_T]^2 +\delta \bm{F}(\delta \bm{E} \cdot \delta \bm{F})-\delta \bm{E} (\delta \bm{F}\cdot \delta \bm{F})\B]d^3\bl\\&+\f12
\s \B[\big(\bm{E}(\bm{F}_L\cdot\bm{F}_L)\big)^\varep-\bm{E}(\bm{F}_L\cdot\bm{F}_L)^\varep\B]
 =  F_j\partial_k\B[(E_k F_{L_j})^\varep-(E_k F_{L_j}^\varep)\B];
\ea\ee
\be\label{2.18}\ba
&\int_{\mathbb{T}^{3}} \nabla \varphi^\varep(\ell)\cdot \delta \bm{E}(\delta \bm{F}_T\cdot \delta \bm{E}_T)-\f{2}{|\bl|}\varphi^\varep\bn\cdot \delta\bm{E} (\delta\bm{F}_T\cdot \delta\bm{E}_T) \\&+\f{1}{|\bl|}\varphi^\varep\bn\cdot\B[\delta \bm{E}(\delta\bm{E}\cdot \delta\bm{F})-\delta \bm{F}(\delta \bm{E}\cdot \delta \bm{E})\B] d^3\bl+\s\B[\big(\bm{E} (\bm{F}_T\cdot\bm{E}_T)\big)^\varep-\bm{E} (\bm{F}_T\cdot \bm{E}_T)^\varep\B]\\
=&E_j\partial_k\B[(E_k F_{T_j})^\varep-(E_k F_{T_j}^\varep)\B]+F_i\partial_k\B[(E_k E_{T_i})^\varep-(E_k E_{T_i}^\varep)\B];
\ea\ee
\be\label{2.19}\ba
&\f12\int_{\mathbb{T}^{3}} \nabla \varphi^\varep(\ell)\cdot \delta \bm{E}[\delta \bm{F}_T]^2-\f{2}{|\bl|}\varphi^\varep\bn\cdot \delta \bm{E} [\delta \bm{F}_T]^2 \\ & -\f{2}{|\bl|}\varphi^\varep\bn\cdot\B[\delta \bm{F}(\delta \bm{E}\cdot \delta \bm{F})-\delta \bm{E}(\delta \bm{F}\cdot \delta \bm{F})\B]d^3\bl+\f12\s \B[\big(\bm{E}(\bm{F}_T\cdot\bm{F}_T)\big)^\varep-\bm{E}(\bm{F}_T\cdot\bm{F}_T)^\varep\B]\\
=&F_j\partial_k\B[(E_k F_{T_j})^\varep-(E_k F_{T_j}^\varep)\B].
\ea\ee
\end{lemma}

\begin{proof}
First, it is simple to check that
\be\label{2.20}\ba
\delta \bm{E}_{L}\cdot \delta \bm{F}_{L}=\bn_i(\bn_j \delta E_j)\cdot\bn_i(\bn_k \delta F_k)=\bn_j \bn_k\delta E_j \delta F_k=\bn_j \bn_i\delta E_j \delta F_i= \bn_i\bn_j\delta E_i \delta F_j,
\ea\ee
\be\label{2.21}\ba
\delta \bm{E}_{T}\cdot \delta \bm{F}_{T}
&=(\delta^{ij}-\bn_i\bn_j)\delta E_j\cdot (\delta^{ik}-\bn_i\bn_k)\delta F_k
&=(\delta^{jk}-\bn_j\bn_k)\delta E_j\delta F_k
&=(\delta^{i j }-\bn_i\bn_j)\delta  E_i\delta F_j,
\ea\ee
and
\be\label{2.22}\ba
\partial_{\ell_k}\bn_i&=
\f{\partial}{\partial\ell_k}\B(\f{\ell_i}{|\bl|}\B)
=\f{\delta^{ik}|\bl|-\f{\ell_i\ell_k}{|\bl|}}{|\bl]^2}
=\f{1}{|\bl|}(\delta^{ik}-\bn_i\bn_k).
\ea\ee

Thanks to \eqref{2.20} and \eqref{2.21}, it follows that
\be\ba
&\int_{\mathbb{T}^{3}} \nabla \varphi^\varep(\ell)\cdot \delta \bm{E}(\bl)(\delta \bm{F}_L \cdot \delta \bm{E}_L)+\f{2}{|\bl|}\varphi^\varep(\ell)\bn\cdot \delta \bm{E}(\bl)(\delta \bm{F}_T\cdot \bm{E}_T) d^3 \bl
\\=&\int_{\mathbb{T}^{3}} \f{\partial \varphi^\varep}{\partial_{\ell_k}}\delta E_k\bn_i\bn_j \delta F_i \delta E_j+ \varphi^\varep\bn_k\cdot \delta E_k(\f{\partial \bn_i}{\partial\ell_j}+\f{\partial \bn_j}{\partial\ell_i}  )\delta F_i \delta E_j d^3 \bl,
\ea\ee
where we have used \eqref{2.22} and
$$\f{2}{|\bl|} (\delta ^{ij}-\bn_i\bn_j)
= \f{\partial \bn_i}{\partial\ell_j}+\f{\partial \bn_j}{\partial\ell_i}.  $$
Then due to the Leibniz formula, we further deduce that
$$ \ba
&\int_{\mathbb{T}^{3}} \nabla \varphi^\varep(\ell)\cdot \delta \bm{E}(\delta \bm{F}_L \cdot \delta \bm{E}_L)+\f{2}{|\bl|}\varphi^\varep(\ell)\bn\cdot \delta \bm{E}(\delta \bm{F}_T\cdot \bm{E}_T) d^3 \bl\\
=&\int_{\mathbb{T}^{3}} \f{\partial }{\partial_{\ell_k}}\B(\varphi^\varep\bn_i\bn_j\B)\delta E_k\delta F_i \delta E_j-\varphi^\varep\B[\f{\partial}{\partial_{\ell_k}}(\bn_i\bn_j)-(\f{\partial \bn_i}{\partial\ell_j}+\f{\partial \bn_j}{\partial\ell_i})\bn_k\B]\delta E_k\delta F_i \delta E_j d^3\bl.
\ea$$
Hence, invoking Lemma \ref{lemma2.1}, we can obtain
\be \ba\label{02.22}
&\int_{\mathbb{T}^{3}} \nabla \varphi^\varep(\ell)\cdot \delta \bm{E}(\delta \bm{F}_L \cdot \delta \bm{E}_L)+\f{2}{|\bl|}\varphi^\varep\bn\cdot \delta \bm{E}(\delta \bm{F}_T\cdot \bm{E}_T) d^3 \bl\\
 =&\int_{\mathbb{T}^{3}} \f{\partial }{\partial_{\ell_k}}\B(\varphi^\varep\bn_i\bn_j\B)\delta E_k\delta F_i \delta E_j+\f{1}{|\bl|}\varphi^\varep\bn\cdot\B[\delta \bm{E}(\delta \bm{E} \cdot \delta \bm{F})-\delta \bm{F} (\delta \bm{E}\cdot \delta \bm{E})\B]d^3\bl,
\ea\ee
which implies
\be \label{2.23}\ba
&\int_{\mathbb{T}^{3}} \nabla \varphi^\varep(\ell)\cdot \delta \bm{E}(\delta \bm{F}_L \cdot \delta \bm{E}_L)+\f{2}{|\bl|}\varphi^\varep\bn\cdot \delta \bm{E}(\delta \bm{F}_T\cdot \bm{E}_T) d^3 \bl\\&- \f{1}{|\bl|}\varphi^\varep\bn\cdot\B[\delta \bm{E}(\delta \bm{E} \cdot \delta \bm{F})-\delta \bm{F} (\delta \bm{E}\cdot \delta \bm{E})\B] d^3 \bl \\&=\int_{\mathbb{T}^{3}} \f{\partial }{\partial_{\ell_k}}\B(\varphi^\varep\bn_i\bn_j\B)\delta E_k\delta F_i \delta E_jd^3\bl.
\ea\ee
 Before going further, we set
$$
\bm{E}(\bx+\bl)=\bar{\bm{E}}=(\bar{E}_1,\bar{E}_2,\bar{E}_3 ),  \bm{F}(\bx+\bl)=\bar{\bm{F}}=(\bar{F}_1,\bar{F}_2,\bar{F}_3 ).
$$
A direct calculation shows
\be\label{2.26}\ba
&\int_{\mathbb{T}^{3}} \f{\partial}{\partial_{\ell_k}}\B(\varphi^\varep \bn_i\bn_j\B)\delta E_k\delta F_i \delta E_jd^3\bl\\
=&\int_{\mathbb{T}^{3}} \f{\partial}{\partial_{\ell_k}}\B[\varphi^\varep \bn_i\bn_j\B]\B[E_k(\bx+\bl)-E_k(\bx)\B]\B[F_i(\bx+\bl)-F_i(\bx)\B]
\B[E_j(\bx+\bl)-E_j(\bx)\B]d^3\bl\\
=&\int_{\mathbb{T}^{3}} \f{\partial}{\partial_{\ell_k}}\B(\varphi^\varep \bn_i\bn_j\B)\\&\times\B(\bar{E}_k \bar{F}_i\bar{E}_j -\bar{E}_k\bar{F}_i E_j-\bar{E}_kF_i\bar{E}_j+\bar{E}_kF_iE_j
-E_k\bar{F}_i\bar{E}_j+E_k\bar{F}_iE_j
+E_kF_i\bar{E}_j-E_kF_iE_j\B)d^3\bl.
\ea\ee
Then we will deal with the terms on the right-hand side of \eqref{2.26} on by one. By using  integration by parts and \eqref{2.20}, we get
\be\label{2.27}\ba
 \int_{\mathbb{T}^{3}} \f{\partial}{\partial_{\ell_k}}\B(\varphi^\varep \bn_i\bn_j\B)\bar{E}_k\bar{F}_i\bar{E}_{j}d^3\bl
=&-\int_{\mathbb{T}^{3}} \varphi^\varep \bn_i\bn_j\f{\partial}{\partial_{\ell_k}}\B(\bar{E}_k\bar{F}_i\bar{E}_{j}\B)d^3\bl\\
=&-\f{\partial}{\partial_{x_k}}\int_{\mathbb{T}^{3}} \varphi^\varep \bn_i\bn_j\bar{E}_k\bar{F}_i\bar{E}_{j}d^3\bl\\
=&-\s\B[\bm{E} (\bm{F}_L\cdot \bm{E}_L)\B]^\varep
\ea\ee
Along the same lines as above, we obtain
\be\label{2.28}\ba
&-\int_{_{\mathbb{T}^{3}}} \f{\partial}{\partial_{\ell_k}}\B[\varphi^\varep \bn_i\bn_j\B]\bar{E}_k\bar{F}_i E_{j}d^3\bl=E_j\partial_k(E_k F_{L_j})^\varep,\\
&-\int_{\mathbb{T}^{3}} \f{\partial}{\partial_{\ell_k}}\B[\varphi^\varep \bn_i\bn_j\B]\bar{E}_kF_i\bar{E}_j d^3\bl=F_i\partial_k(E_k E_{L_i})^\varep,\\
&\int_{\mathbb{T}^{3}} \f{\partial}{\partial_{\ell_k}}\B[\varphi^\varep \bn_i\bn_j\B]\bar{E}_kF_iE_j d^3\bl=-F_iE_j\int_{\mathbb{T}^{3}} \f{\partial}{\partial_{\ell_k}}\B[\varphi^\varep \bn_i\bn_j\B]\bar{E}_k d^3\bl=0,\\
&-\int_{\mathbb{T}^{3}} \f{\partial}{\partial_{\ell_k}}\B[\varphi^\varep \bn_i\bn_j\B]E_k\bar{F}_i\bar{E}_jd^3\bl=E_k\partial_k(\bm{F}_L\cdot \bm{E}_L)^\varep=\s [\bm{E}(\bm{F}_L\cdot \bm{E}_L)^\varep],\\
&\int_{\mathbb{T}^{3}} \f{\partial}{\partial_{\ell_k}}\B[\varphi^\varep \bn_i\bn_j\B]E_k\bar{F}_iE_jd^3\bl =
-E_k E_j\partial_k(F_{L_j})^\varep=-E_j\partial_k(E_k F_{L_j}^\varep),\\
&\int_{\mathbb{T}^{3}} \f{\partial}{\partial_{\ell_k}}\B[\varphi^\varep \bn_i\bn_j\B]E_kF_i\bar{E}_jd^3 \bl =-E_k F_i\partial_k E_{L_i}^\varep=-F_i\partial_k(E_k E_{L_i}^\varep),\\
&-\int_{\mathbb{T}^{3}} \f{\partial}{\partial_{\ell_k}}\B[\varphi^\varep \bn_i\bn_j\B]E_kF_iE_j d^3\bl =0.
\ea\ee
Consequently, we see that
\be\ba\label{2.29}
&\int_{\mathbb{T}^{3}} \f{\partial}{\partial_{\ell_k}}\B(\varphi^\varep \bn_i\bn_j\B)\delta E_k\delta F_i \delta E_jd^3\bl\\
=&-\s\B[\big(\bm{E} (\bm{F}_L\cdot \bm{E}_L)\big)^\varep-\bm{E}(\bm{F}_L\cdot \bm{E}_L)^\varep\B]\\
&+E_j\partial_k\B[(E_k F_{L_j})^\varep-(E_k F_{L_j}^\varep)\B]+F_i\partial_k\B[(E_k E_{L_i})^\varep-(E_k E_{L_i}^\varep)\B].
\ea\ee
Inserting \eqref{2.29} into \eqref{2.23}, we arrive at
\be\label{2.30}\ba
&\int_{\mathbb{T}^{3}} \nabla \varphi^\varep(\ell)\cdot \delta \bm{E}(\delta \bm{F}_L \cdot \delta \bm{E}_L)+\f{2}{|\bl|}\varphi^\varep\bn\cdot \delta \bm{E}(\delta \bm{F}_T\cdot \bm{E}_T)\\&- \f{1}{|\bl|}\varphi^\varep\bn\cdot\B[\delta \bm{E}(\delta \bm{E} \cdot \delta \bm{F})-\delta \bm{F} (\delta \bm{E}\cdot \delta \bm{E})\B] d^3 \bl+\s\B[\big(\bm{E} (\bm{F}_L\cdot \bm{E}_L)\big)^\varep-\bm{E}(\bm{F}_L\cdot \bm{E}_L)^\varep\B]\\ &=E_j\partial_k\B[(E_k F_{L_j})^\varep-(E_k F_{L_j}^\varep)\B]+F_i\partial_k\B[(E_k E_{L_i})^\varep-(E_k E_{L_i}^\varep)\B].
\ea\ee
Then we have proved \eqref{2.16}.

To obtain the \eqref{2.17}, following the same path of \eqref{02.22}, we find
 \be\label{2.32}\ba
&\int_{\mathbb{T}^{3}} \nabla \varphi^\varep(\ell)\cdot \delta \bm{E} [\delta \bm{F}_L]^2+\f{2}{|\bl|}\varphi^\varep\bn\cdot \delta \bm{E} [\delta \bm{F}_T]^2 d^3\bl\\
=&\int_{\mathbb{T}^{3}} \f{\partial \varphi^\varep}{\partial_{\ell_k}}\cdot \delta E_k\bn_i\bn_j\delta F_i \delta F_j+\varphi^\varep\bn_k\delta E_k (\f{\partial \bn_i}{\partial\ell_j}+\f{\partial \bn_j}{\partial\ell_i})\delta F_i\delta F_j d^3\bl\\
	=&\int_{\mathbb{T}^{3}} \f{\partial}{\partial_{\ell_k}}\B(\varphi^\varep \bn_i\bn_j\B)\delta E_k\delta F_i\delta F_j -\varphi^\varep\B[\f{\partial}{\partial_{\ell_k}}(\bn_i\bn_j)-(\f{\partial \bn_i}{\partial\ell_j}+\f{\partial \bn_j}{\partial\ell_i})\bn_k\B]\delta E_k \delta F_i\delta F_j d^3\bl\\
	=&\int_{\mathbb{T}^{3}} \f{\partial}{\partial_{\ell_k}}\B(\varphi^\varep \bn_i\bn_j\B)\delta E_k\delta F_i\delta F_j -\f{2}{|\bl|}\varphi^\varep\bn\cdot\B[\delta \bm{F}(\delta \bm{E} \cdot \delta \bm{F})-\delta \bm{E} (\delta \bm{F}\cdot \delta \bm{F})\B]d^3\bl.
\ea\ee
A slight modification of deduction of  \eqref{2.29}, we conclude
that
\be\label{2.33}\ba
 &\int_{\mathbb{T}^{3}} \f{\partial}{\partial_{\ell_k}}\B[\varphi^\varep \bn_i\bn_j\B]\delta E_k\delta F_i\delta F_j  d^3\bl\\
=&-\s \B[\big(\bm{E}(\bm{F}_L\cdot\bm{F}_L)\big)^\varep-\bm{E}(\bm{F}_L\cdot\bm{F}_L)^\varep\B]+2F_j\partial_k\B[(E_k F_{L_j})^\varep-(E_k F_{L_j}^\varep)\B].
\ea\ee
Substituting this into \eqref{2.32}, we arrive at
\be\label{2.330}\ba
&\int_{\mathbb{T}^{3}} \nabla \varphi^\varep(\ell)\cdot \delta \bm{E} [\delta \bm{F}_L]^2+\f{2}{|\bl|}\varphi^\varep\bn\cdot \B[\delta \bm{E} [\delta \bm{F}_T]^2 +\delta \bm{F}(\delta \bm{E} \cdot \delta \bm{F})-\delta \bm{E} (\delta \bm{F}\cdot \delta \bm{F})\B]d^3\bl\\
=&-\s \B[\big(\bm{E}(\bm{F}_L\cdot\bm{F}_L)\big)^\varep-\bm{E}(\bm{F}_L\cdot\bm{F}_L)^\varep\B]+2F_j\partial_k\B[(E_k F_{L_j})^\varep-(E_k F_{L_j}^\varep)\B],
\ea\ee
which  yields that
\be\label{2.34}\ba
&\f12\int_{\mathbb{T}^{3}} \nabla \varphi^\varep(\ell)\cdot \delta \bm{E} [\delta \bm{F}_L]^2+\f{2}{|\bl|}\varphi^\varep\bn\cdot \B[\delta \bm{E} [\delta \bm{F}_T]^2 +\delta \bm{F}(\delta \bm{E} \cdot \delta \bm{F})-\delta \bm{E} (\delta \bm{F}\cdot \delta \bm{F})\B]d^3\bl\\&+\f12
\s \B[\big(\bm{E}(\bm{F}_L\cdot\bm{F}_L)\big)^\varep-\bm{E}(\bm{F}_L\cdot\bm{F}_L)^\varep\B]
\\=& F_j\partial_k\B[(E_k F_{L_j})^\varep-(E_k F_{L_j}^\varep)\B].
\ea\ee
Then the desired equality \eqref{2.17} has been proved.

Regarding the rest equalities \eqref{2.18} and \eqref{2.19}, repeating the derivation of \eqref{02.22} and replacing the application of \eqref{2.20} by
\eqref{2.21}, we arrive at
\be\label{2.35}\ba
&\int_{\mathbb{T}^{3}} \nabla \varphi^\varep(\ell)\cdot \delta \bm{E}(\delta \bm{F}_T\cdot \delta \bm{E}_T)-\f{2}{|\bl|}\varphi^\varep\bn\cdot \delta\bm{E} (\delta\bm{F}_T\cdot \delta\bm{E}_T) d^3\bl\\
=&\int_{\mathbb{T}^{3}} \partial_{\ell_k}\varphi^\varep\cdot\delta E_k (\delta^{ij}-\bn_i\bn_j)\delta F_i\delta E_j -\f{2}{|\bl|}\varphi^\varep\bn_k \cdot\delta E_k(\delta^{ij}-\bn_i\bn_j)\delta F_i \delta E_jd^3\bl\\
=&\int_{\mathbb{T}^{3}} \partial_{\ell_k}\B[\varphi^\varep(\delta^{ij}-\bn_i\bn_j)\B]\delta E_k\delta F_i \delta E_j+\varphi^\varep\B[\partial_{\ell_k}(\bn_i\bn_j)-(\f{\partial \bn_i}{\partial_{\ell_j}}+\f{\partial \bn_j}{\partial_{\ell_i}})\bn_k\B]\delta E_k\delta F_i \delta E_j d^3\bl\\
=&\int_{\mathbb{T}^{3}} \partial_{\ell_k}\B[\varphi^\varep(\delta^{ij}-\bn_i\bn_j)\B]\delta E_k\delta F_i \delta E_j-\f{1}{|\bl|}\varphi^\varep\bn\cdot\B[\delta \bm{E}(\delta\bm{E}\cdot \delta\bm{F})-\delta \bm{F}(\delta \bm{E}\cdot \delta \bm{E})\B] d^3\bl,
\ea\ee
Using the preceding calculations in \eqref{2.26}, we obtain
$$\ba
&\int_{\mathbb{T}^{3}} \partial_{\ell_k}\B[\varphi^\varep(\delta^{ij}-\bn_i\bn_j)\B]\delta E_k\delta F_i \delta E_j d^3\bl\\
=&\int_{\mathbb{T}^{3}} \partial_{\ell_k}\B[\varphi^\varep(\delta^{ij}-\bn_i\bn_j)\B]
\B[\overline{E}_k\overline{F}_{i}\overline{E}_{j}
-\overline{E}_k\overline{F}_{i} {E}_{j}
-\overline{E}_k {F}_{i}\overline{E}_{j}
+\overline{E}_k {F}_{i} {E}_{j}\\
&~~~~~~~~~~~~~~~~~~~~~~~~~~~~~~~~~
- {E}_k\overline{F}_{i}\overline{E}_{j}
+ {E}_k\overline{F}_{i} {E}_{j}
+ {E}_k {F}_{i}\overline{E}_{j}
- {E}_k {F}_{i} {E}_{j} \B]d^3\bl
\ea$$
Integrating by parts, we conclude by \eqref{2.21} that
\be\label{2.36}\ba
 \int_{\mathbb{T}^{3}} \f{\partial}{\partial_{\ell_k}} \B[\varphi^\varep(\delta^{ij}-\bn_i\bn_j)\B]\bar{E}_k\bar{F}_i\bar{E}_{j}
d^3\bl=&-\int_{\mathbb{T}^{3}} \partial_{\ell_k}\B[\varphi^\varep(\delta^{ij}-\bn_i\bn_j)\B]\f{\partial}{\partial_{\ell_k}}\B(\bar{E}_k\bar{F}_i\bar{E}_{j}\B)d^3\bl\\
=&- \int_{\mathbb{T}^{3}} \partial_{\ell_k}\B[\varphi^\varep(\delta^{ij}-\bn_i\bn_j)\B]\bar{E}_k\bar{F}_i\bar{E}_{j}d^3\bl\\
=&-\s\B[\bm{E} (\bm{F}_T\cdot \bm{E}_T)\B]^\varep.
\ea\ee
By the same token, we also have
\begin{align}
&-\int_{_{\mathbb{T}^{3}}} \f{\partial}{\partial_{\ell_k}}\B[\varphi^\varep(\delta^{ij}-\bn_i\bn_j)\B]\bar{E}_k\bar{F}_i E_{j}d^3\bl=E_j\partial_k(E_k F_{T_j})^\varep,\nonumber\\
&-\int_{\mathbb{T}^{3}} \f{\partial}{\partial_{\ell_k}}\B[\varphi^\varep(\delta^{ij}-\bn_i\bn_j)\B]\bar{E}_kF_i\bar{E}_j d^3\bl=F_i\partial_k(E_k E_{T_i})^\varep,\nonumber\\
&\int_{\mathbb{T}^{3}} \f{\partial}{\partial_{\ell_k}}\B[\varphi^\varep(\delta^{ij}-\bn_i\bn_j)\B]\bar{E}_kF_iE_j d^3\bl=-F_iE_j\int_{\mathbb{T}^{3}} \f{\partial}{\partial_{\ell_k}}\B[\varphi^\varep(\delta^{ij}-\bn_i\bn_j)\B]\bar{E}_k d^3\bl=0,\nonumber\\
&-\int_{\mathbb{T}^{3}} \f{\partial}{\partial_{\ell_k}}\B[\varphi^\varep(\delta^{ij}-\bn_i\bn_j)\B]E_k\bar{F}_i\bar{E}_jd^3\bl
=E_k\partial_k(\bm{F}_T\cdot \bm{E}_T)^\varep
=\s [\bm{E}(\bm{F}_T\cdot \bm{E}_T)^\varep],\nonumber\\
&\int_{\mathbb{T}^{3}} \f{\partial}{\partial_{\ell_k}}
\B[\varphi^\varep(\delta^{ij}-\bn_i\bn_j)\B]E_k\bar{F}_iE_jd^3\bl =
-E_k E_j\partial_k(F_{T_j})^\varep=-E_j\partial_k(E_k F_{T_j}^\varep),\nonumber\\
&\int_{\mathbb{T}^{3}} \f{\partial}{\partial_{\ell_k}}\B[\varphi^\varep(\delta^{ij}-\bn_i\bn_j)\B]E_kF_i\bar{E}_jd^3 \bl =-E_k F_i\partial_k E_{T_i}^\varep=-F_i\partial_k(E_k E_{T_i}^\varep),\nonumber\\
&-\int_{\mathbb{T}^{3}} \f{\partial}{\partial_{\ell_k}}\B[\varphi^\varep(\delta^{ij}-\bn_i\bn_j)\B]E_kF_iE_j d^3\bl =0.\nonumber
\end{align}
As a consequence, we find
\be \ba\label{02.39}
 &\int_{\mathbb{T}^{3}} \partial_{\ell_k}\B[\varphi^\varep(\delta^{ij}-\bn_i\bn_j)\B]\delta E_k\delta F_i \delta E_jd^3\bl\\
=&-\s\B[\big(\bm{E} (\bm{F}_T\cdot \bm{E}_T)\big)^\varep-\bm{E}(\bm{F}_T\cdot \bm{E}_T)^\varep\B]\\
&+E_j\partial_k\B[(E_k F_{T_j})^\varep-(E_k F_{T_j}^\varep)\B]+F_i\partial_k\B[(E_k E_{T_i})^\varep-(E_k E_{T_i}^\varep)\B].
\ea\ee
Inserting this into \eqref{2.35}, we know that
\be\label{2.39}\ba
&\int_{\mathbb{T}^{3}} \nabla \varphi^\varep(\ell)\cdot \delta \bm{E}(\delta \bm{F}_T\cdot \delta \bm{E}_T)-\f{2}{|\bl|}\varphi^\varep\bn\cdot \delta\bm{E} (\delta\bm{F}_T\cdot \delta\bm{E}_T)\\&+\f{1}{|\bl|}\varphi^\varep\bn\cdot\B[\delta \bm{E}(\delta\bm{E}\cdot \delta\bm{F})-\delta \bm{F}(\delta \bm{E}\cdot \delta \bm{E})\B] d^3\bl+\s\B[\big(\bm{E} (\bm{F}_T\cdot \bm{E}_T)\big)^\varep-\bm{E}(\bm{F}_T\cdot \bm{E}_T)^\varep\B]\\
=&E_j\partial_k\B[(E_k F_{T_j})^\varep-(E_k F_{T_j}^\varep)\B]+F_i\partial_k\B[(E_k E_{T_i})^\varep-(E_k E_{T_i}^\varep)\B].
\ea\ee
Thus the validity of \eqref{2.18} is  confirmed.

Finally, proceeding as in the proof of
\eqref{2.35}, we have
\be\label{2.40}\ba
&\int_{\mathbb{T}^{3}} \nabla \varphi^\varep(\ell)\cdot \delta \bm{E}[\delta \bm{F}_T]^2-\f{2}{|\bl|}\varphi^\varep\bn\cdot \delta \bm{E} [\delta \bm{F}_T]^2 d^3\bl\\
=&\int_{\mathbb{T}^{3}} \partial_{\ell_k}\varphi^\varep\cdot\delta E_k (\delta^{ij}-\bn_i\bn_j)\delta F_i \delta F_j -\f{2}{|\bl|}\varphi^\varep\bn_k \cdot\delta E_k(\delta^{ij}-\bn_i\bn_j)\delta F_i \delta F_j d^3\bl\\
=&\int_{\mathbb{T}^{3}} \partial_{\ell_k}\B[\varphi^\varep(\delta^{ij}-\bn_i\bn_j)\B]\delta E_k\delta F_i \delta F_j+\varphi^\varep\B[\partial_{\ell_k}(\bn_i\bn_j)-(\f{\partial \bn_i}{\partial_{\ell_j}}+\f{\partial \bn_j}{\partial_{\ell_i}})\bn_k\B]\delta E_k\delta F_i \delta F_j d^3\bl\\
=&\int_{\mathbb{T}^{3}} \partial_{\ell_k}\B[\varphi^\varep(\delta^{ij}-\bn_i\bn_j)\B]\delta E_k\delta F_i \delta F_j+\f{2}{|\bl|}\varphi^\varep\bn\cdot\B[\delta \bm{F}(\delta \bm{E}\cdot \delta \bm{F})-\delta \bm{E}(\delta \bm{F}\cdot \delta \bm{F})\B] d^3\bl.
\ea\ee
A slight variant of the   proof of \eqref{02.39} provides
\be\label{2.41}\ba
& \int_{\mathbb{T}^{3}} \partial_{\ell_k}\B[\varphi^\varep(\delta^{ij}-\bn_i\bn_j)\B]\delta E_k\delta F_i \delta F_j d^3\bl\\
=&-\s \B[\big(\bm{E}(\bm{F}_T\cdot\bm{F}_T)\big)^\varep-\bm{E}(\bm{F}_T\cdot\bm{F}_T)^\varep\B]+2F_j\partial_k\B[(E_k F_{T_j})^\varep-(E_k F_{T_j}^\varep)\B].
\ea\ee
A combination of this and \eqref{2.40}, we end up with
\be\label{2.42}\ba
&\f12\int_{\mathbb{T}^{3}} \nabla \varphi^\varep(\ell)\cdot \delta \bm{E}[\delta \bm{F}_T]^2-\f{2}{|\bl|}\varphi^\varep\bn\cdot \B[\delta \bm{E} [\delta \bm{F}_T]^2  +\delta \bm{F}(\delta \bm{E}\cdot \delta \bm{F})-\delta \bm{E}(\delta \bm{F}\cdot \delta \bm{F})\B]d^3\bl\\&+\f12\s \B[\big(\bm{E}(\bm{F}_T\cdot\bm{F}_T)\big)^\varep-\bm{E}(\bm{F}_T\cdot\bm{F}_T)^\varep\B]
=F_j\partial_k\B[(E_k F_{T_j})^\varep-(E_k F_{T_j}^\varep)\B].
\ea\ee
This achieves the proof of this lemma.
\end{proof}

\section{Exact scaling laws of energy in the EMHD system}
We are in position to start the proof four-fifths laws of the energy to the inviscid EMHD equations.
 \begin{proof}[Proof of Theorem \ref{the1.1}]
Thanks to the definition of \eqref{defiofeyink} and \eqref{EMHD}, we know that
\be\label{3.1}
 \partial_{t}\bm{b}^{\varepsilon}_{L}+
 \text{d}_{\text{I}}\nabla\times[(\nabla\times \bm{b})\times \bm{b}_{L}]^{\varepsilon} =0,\s\bm{b}_{L}=0.
 \ee
Multiplying \eqref{3.1} and \eqref{EMHD} by $\bm{b}$ and $ \bm{b}^{\varepsilon}_{L}$, respectively, we obtain
\be\label{3.2}
 \partial_{t}\bm{b}^{\varepsilon}_{L} \cdot\bm{b} +
 \text{d}_{\text{I}}\nabla\times[(\nabla\times \bm{b})\times \bm{b}_{L}]^{\varepsilon}\cdot\bm{b} =0,\s\bm{b}=0,
 \ee
 \be\label{3.3}
 \partial_{t}\bm{b} \cdot\bm{b}^{\varepsilon}_{L} +
 \text{d}_{\text{I}}\nabla\times[(\nabla\times \bm{b})\times \bm{b}]\cdot\bm{b}_{L}^{\varepsilon} =0,\s\bm{b}=0.
 \ee
By the sum \eqref{3.2} and \eqref{3.3}, we arrive at
 \be\label{3.4}
 \partial_{t}(\bm{b} \cdot\bm{b}^{\varepsilon}_{L})
 +
 \text{d}_{\text{I}}\nabla\times[(\nabla\times \bm{b})\times \bm{b}_{L}]^{\varepsilon}\cdot\bm{b} +
 \text{d}_{\text{I}}\nabla\times[(\nabla\times \bm{b})\times \bm{b}]\cdot\bm{b}_{L}^{\varepsilon} =0
\ee
Notice that
$$\ba &\nabla\times[(\nabla\times \bm{b})\times \bm{b}]\cdot \bm{b}^{\varepsilon}_{L}+ \nabla\times[(\nabla\times \bm{b})\times \bm{b}_{L}]^{\varepsilon}\cdot \bm{b}\\
= &\f12\{2 \nabla\times[(\nabla\times \bm{b})\times \bm{b}]\cdot \bm{b}^{\varepsilon}_{L}+2 \nabla\times[(\nabla\times \bm{b})\times \bm{b}_{L}]^{\varepsilon}\cdot \bm{b}\}.
\ea$$
According to \eqref{non1} and \eqref{non2}, we have
\begin{align}\label{3.5}
 &\nabla\times[(\nabla\times \bm{b})\times \bm{b}_{L}]^{\varepsilon}=
 \s( \bm{b}\otimes\bm{J}_{L})^{\varepsilon}- \s(\bm{J}\otimes  \bm{b}_{L})^{\varepsilon},\\
 \label{3.6}&\nabla\times[(\nabla\times \bm{b})\times \bm{b}_{L}]^{\varepsilon} =\nabla\times[\s(\bm{b} \otimes \bm{b}_{L})]^{\varepsilon}.
\end{align}
In view of \eqref{non1}-\eqref{non4} and $\eqref{VI}_{3}$, we infer that
$$\ba& 2 \nabla\times[(\nabla\times \bm{b})\times \bm{b}]\cdot \bm{b}_{L}^{\varepsilon}\\=&[\s(\bm{b}\otimes\bm{J})- \s(\bm{J}\otimes \bm{b})]\cdot \bm{b}_{L}^{\varepsilon}+\nabla\times[\s(\bm{b} \otimes \bm{b})]\cdot \bm{b}_{L}^{\varepsilon}\\=&
[\s(\bm{b}\otimes\bm{J})- \s(\bm{J}\otimes \bm{b})]\cdot \bm{b}_{L}^{\varepsilon}+\s([\s(\bm{b} \otimes \bm{b})]\times \bm{b}_{L}^{\varepsilon})+[\s(\bm{b} \otimes \bm{b})]\cdot \bm{J}^{\varepsilon}_{L}.
\ea$$
and
$$\ba
&2 \nabla\times[(\nabla\times \bm{b})\times \bm{b}_{L}]^{\varepsilon}\cdot \bm{b}\\=&[\s(\bm{b}\otimes\bm{J}_{L})- \s(\bm{J}\otimes \bm{b}_{L})]^{\varepsilon}\cdot \bm{b}+\nabla\times[\s(\bm{b} \otimes \bm{b}_{L})]^{\varepsilon}\cdot \bm{b}\\
=&[\s(\bm{b}\otimes\bm{J}_{L})- \s(\bm{J}\otimes \bm{b}_{L})]^{\varepsilon}\cdot \bm{b}+\s([\s(\bm{b} \otimes \bm{b}_{L})]^{\varepsilon}\times \bm{b})+[\s(\bm{b} \otimes \bm{b}_{L})]^{\varepsilon} \cdot \bm{J}.
\ea$$
Hence, the second and third terms on the left-hand side of \eqref{3.4} can be rewritten as
\be\ba\label{3.7}
&\text{d}_{\text{I}}\big\{\nabla\times[(\nabla\times \bm{b})\times \bm{b}]\cdot \bm{b}_{L}^{\varepsilon}+ \nabla\times[(\nabla\times \bm{b})\times \bm{b}_{L}]^{\varepsilon}\cdot \bm{b}\big\}\\
=&\f{\text{d}_{\text{I}}}{2}\B\{\s([\s(\bm{b} \otimes \bm{b}_{L})]^{\varepsilon}\times \bm{b})+\s([\s(\bm{b} \otimes \bm{b})]\times \bm{b}^{\varepsilon}_{L})\\&
+\s(\bm{b}\otimes\bm{J})\cdot \bm{b}_{L}^{\varepsilon}+[\s(\bm{b} \otimes \bm{b})]\cdot \bm{J}_{L}^{\varepsilon}+\s(\bm{b}\otimes\bm{J}_{L})^{\varepsilon}\cdot \bm{b}+[\s(\bm{b} \otimes \bm{b}_{L})]^{\varepsilon} \cdot \bm{J}\\&- \s(\bm{J}\otimes \bm{b}_{L})^{\varepsilon}\cdot \bm{b}- \s(\bm{J}\otimes \bm{b})\cdot \bm{b}_{L}^{\varepsilon}\B\}.
\ea\ee
Then we plug  \eqref{3.7} into \eqref{3.4} to obtain that
 \be\ba \label{3.8}
& \partial_{t}( \bm{b}\cdot\bm{b}_{L}^{\varepsilon})+\f{\text{d}_{\text{I}}}{2}\B\{\s([\s(\bm{b} \otimes \bm{b}_{L})]^{\varepsilon}\times \bm{b})+\s([\s(\bm{b} \otimes \bm{b})]\times \bm{b}_{L}^{\varepsilon})\\&+[\s(\bm{b} \otimes \bm{b}_{L})]^{\varepsilon}\cdot \bm{J}+[\s(\bm{b} \otimes \bm{b})]\cdot \bm{J}_{L}^{\varepsilon}+ \s( \bm{b}\otimes\bm{J}_{L})^{\varepsilon} \cdot \bm{b}+ \s( \bm{b}\otimes\bm{J})\cdot\bm{b}_{L}^{\varepsilon}\\&- \s(\bm{J}\otimes \bm{b}_{L})^{\varepsilon}\cdot \bm{b}- \s(\bm{J}\otimes  \bm{b})\cdot \bm{b}_{L}^{\varepsilon}\B\} =0.
\ea\ee
By takig a straightforward computation and using the divergence-free condition, it is easy to check that
\be\ba\label{3.9}
\partial_{k}(b_{k}b_{L_{i}})^{\varepsilon}J_{i}+\partial_{k}(b_{k}b_{i}) J_{L_{i}}^{\varepsilon}=&\partial_{k}(b_{k}b_{i}J_{L_{i}}^{\varepsilon}) +\partial_{k}(b_{k}b_{L_{i}})^{\varepsilon}J_{i}-(b_{k}b_{i}) \partial_{k}J_{L_{i}}^{\varepsilon}\\=&\s\big[\bm{b}(\bm{b}\cdot \bm{J}^{\varepsilon}_L)\big] +J_{i}\partial_{k}(b_{k}b_{L_{i}})^{\varepsilon}- b_{i}  \partial_{k}(b_{k}J_{L_{i}}^{\varepsilon}),\\
\partial_{k}(b_{k}J_{L_{i}})^{\varepsilon}b_{i}+\partial_{k}(b_{k}J_{i}) b_{L_{i}}^{\varepsilon}=&
\partial_{k}(b_{k}J_{i}b _{L_{i}}^{\varepsilon})+\partial_{k}(b_{k}J_{L_{i}})^{\varepsilon}b_{i}-(b_{k}J_{i}) \partial_{k}b_{L_{i}}^{\varepsilon}\\=&
\s\big[\bm{b}(\bm{J}\cdot\bm{b}_L^\varep)\big]+b_{i}\partial_{k}(b_{k}J_{L_{i}})^{\varepsilon}-J_{i} \partial_{k}(b_{k}b_{L_{i}}^{\varepsilon})\\
-\B(\partial_{k}(J_{k}b_{L_{i}})^{\varepsilon}b_{i}+\partial_{k}(J_{k}b_{i})b_{L_{i}}^{\varepsilon}\B)
=&-\partial_{k}(J_{k}b_{i}b_{L_{i}}^{\varepsilon})+(J_{k}b_{i})\partial_{k}b_{L_{i}}^{\varepsilon}-
\partial_{k}(J_{k}b_{L_{i}})^{\varepsilon}b_{i}\\
=&-\s\big[\bm{J}(\bm{b}\cdot\bm{b}_L)^\varep\big]-b_{i}
\partial_{k}\big[(J_{k}b_{L_{i}})^{\varepsilon}- (J_{k}b_{L_{i}}^{\varepsilon})\big].\\
\ea\ee
Substituting \eqref{3.9} into \eqref{3.8}, we see that
 \be\ba\label{3.10}
 & \partial_{t}(\bm{b}_{L}^{\varepsilon}\cdot\bm{b})+\f{\text{d}_{\text{I}}}{2}\s([\s(\bm{b} \otimes \bm{b}_{L})]^{\varepsilon}\times \bm{b})+\f{\text{d}_{\text{I}}}{2}\s([\s(\bm{b} \otimes \bm{b})]\times \bm{b}_{L}^{\varepsilon})\\&+\f{\text{d}_{\text{I}}}{2}\s\big[\bm{b} (\bm{b}\cdot \bm{J}_{L}^{\varepsilon}) +\bm{b} (\bm{b}_{L}^{\varepsilon}\cdot \bm{J})- \bm{J} (\bm{b}\cdot \bm{b}_{L}^{\varepsilon})\big]
 \\&=\f{\text{d}_{\text{I}}}{2}b_{i}
 \partial_{k}\big[(J_{k}b_{L_{i}})^{\varepsilon}- (J_{k}b_{L_{i}}^{\varepsilon})\big]\\&~~-\f{\text{d}_{\text{I}}}{2}J_{i}\partial_{k}[(b_{k}b_{L_{i}})^{\varepsilon}-(b_{k}b_{L_{i}}^{\varepsilon})]-\f{\text{d}_{\text{I}}}{2}b_{i}\partial_{k}\big[(b_{k}J_{L_{i}})^{\varepsilon}-(b_{k}J_{L_{i}}^{\varepsilon})\big].
\ea\ee
On the other hand, by means of \eqref{2.16} and \eqref{2.17} in  Lemma  \ref{lemma2.2}, it follows that
\be\label{3.11}\ba
&\int_{\mathbb{T}^{3}} \nabla \varphi^\varep(\ell)\cdot \delta \bm{b}(\delta \bm{J}_L \cdot \delta \bm{b}_L)+\f{2}{|\bl|}\varphi^\varep\bn\cdot \delta \bm{b}(\bl)(\delta \bm{J}_T\cdot \delta\bm{b}_T)\\&- \f{1}{|\bl|}\varphi^\varep\bn\cdot\B[\delta \bm{b}(\delta \bm{b} \cdot \delta \bm{J})-\delta \bm{J} (\delta \bm{b}\cdot \delta \bm{b})\B] d^3 \bl +\s\B[\big(\bm{b}(\bm{J}_L\cdot \bm{b}_L)\big)^\varep-\bm{b}(\bm{J}_L\cdot \bm{b}_L)^\varep\B]\\&=b_j\partial_k\B[(b_k J_{L_j})^\varep-(b_k J_{L_j}^\varep)\B]+J_i\partial_k\B[(b_k b_{L_i})^\varep-J_i\partial_k(b_k b_{L_i}^\varep)\B],
\ea\ee
  and
  \be\label{3.12}\ba
&\f12\int_{\mathbb{T}^{3}} \nabla \varphi^\varep(\ell)\cdot \delta \bm{J}[\delta \bm{b}_L]^2+\f{2}{|\bl|}\varphi^\varep\bn\cdot \B[\delta \bm{J} [\delta \bm{b}_T]^2 +\delta \bm{b}(\delta \bm{J} \cdot \delta \bm{b})-\delta \bm{J} (\delta \bm{b}\cdot \delta \bm{b})\B]d^3\bl\\&+\f12
\s \B[\big(\bm{J}(\bm{b}_L\cdot\bm{b}_L)\big)^\varep-\bm{J}(\bm{b}_L\cdot\bm{b}_L)^\varep\B]
= b_j\partial_k\B[(J_k b_{L_j})^\varep-(J_k b_{L_j}^\varep)\B].
\ea\ee
 Inserting \eqref{3.11} and \eqref{3.12} into \eqref{3.10}, we have
\be\ba\label{3.14}
 & \partial_{t}(\bm{b}_{L}^{\varepsilon}\cdot\bm{b})+\f{\text{d}_{\text{I}}}{2}\s([\s(\bm{b} \otimes \bm{b}_{L})]^{\varepsilon}\times \bm{b})+\f{\text{d}_{\text{I}}}{2}\s([\s(\bm{b} \otimes \bm{b})]\times \bm{b}_{L}^{\varepsilon})\\
 &+\f{\text{d}_{\text{I}}}{2}\s\big[\bm{b} (\bm{b}\cdot \bm{J}_{L}^{\varepsilon}) +\bm{b} (\bm{b}_{L}^{\varepsilon}\cdot \bm{J})- \bm{J} (\bm{b}\cdot \bm{b}_{L}^{\varepsilon})\big]\\&+\f{\text{d}_{\text{I}}}{2}\s\B[\big(\bm{b}(\bm{J}_L\cdot \bm{b}_L)\big)^\varep-\bm{b}(\bm{J}_L\cdot \bm{b}_L)^\varep\B]
-
\f{\text{d}_{\text{I}}}{4}
\s \B[\big(\bm{J}(\bm{b}_L\cdot\bm{b}_L)\big)^\varep-\bm{J}(\bm{b}_L\cdot\bm{b}_L)^\varep\B]\\
 =&
 -\f{2}{3}D_{EL}^{\varepsilon}(\bm{b},\bm{J}),
\ea\ee
where
$$\ba
&D_{EL}^{\varepsilon}(\bm{b},\bm{J})\\
=&\f{3\text{d}_{\text{I}}}{4 }\int_{\mathbb{T}^{3}} \nabla \varphi^\varep(\ell)\cdot \delta \bm{b}(\delta \bm{J}_L \cdot \delta \bm{b}_L)+\f{2}{|\bl|}\varphi^\varep\B[\bn\cdot \delta \bm{b}(\delta \bm{J}_T\cdot \delta \bm{b}_T)+\delta\bm{b} \times(\delta\bm{J}\times\delta\bm{b})\B]  d^3 \bl\\
&-\f{3\text{d}_{\text{I}}}{8 }\int_{\mathbb{T}^{3}} \nabla \varphi^\varep(\ell)\cdot \delta \bm{J} [\delta \bm{b}_L]^2+\f{2}{|\bl|}\varphi^\varep\bn\cdot  \delta \bm{J}[\delta \bm{b}_T]^2  d^3\bl.
\ea$$
Here, we used vector triple product formula
$$\delta\bm{b} \times(\delta\bm{J}\times\delta\bm{b})=\delta \bm{J} (\delta \bm{b}\cdot \delta \bm{b})-\delta \bm{b}(\delta \bm{J} \cdot \delta \bm{b}).$$

Next, following the path of \eqref{3.10}, we get
\be\ba\label{3.15}
  & \partial_{t}(\bm{b}_{T}^{\varepsilon}\cdot\bm{b})+\f{\text{d}_{\text{I}}}{2}\s([\s(\bm{b} \otimes \bm{b}_{T})]^{\varepsilon}\times \bm{b})+\f{\text{d}_{\text{I}}}{2}\s([\s(\bm{b} \otimes \bm{b})]\times \bm{b}_{T}^{\varepsilon})\\&+\f{\text{d}_{\text{I}}}{2}\s\big[\bm{b} (\bm{b}\cdot \bm{J}_{T}^{\varepsilon}) +\bm{b} (\bm{b}_{T}^{\varepsilon}\cdot \bm{J})- \bm{J} (\bm{b}\cdot \bm{b}_{T}^{\varepsilon})\big]
 \\&=\f{\text{d}_{\text{I}}}{2}b_{i}
 \partial_{k}\big[(J_{k}b_{T_{i}})^{\varepsilon}- (J_{k}b_{T_{i}}^{\varepsilon})\big]\\&~~-\f{\text{d}_{\text{I}}}{2}J_{i}\partial_{k}[(b_{k}b_{T_{i}})^{\varepsilon}-(b_{k}b_{T_{i}}^{\varepsilon})]-\f{\text{d}_{\text{I}}}{2}b_{i}\partial_{k}\big[(b_{k}J_{T_{i}})^{\varepsilon}-(b_{k}J_{T_{i}}^{\varepsilon})\big].
\ea\ee
Besides, in light of  \eqref{2.18} and \eqref{2.19} in Lemma \ref{lemma2.2}, we can deduce that
\be\label{3.16}\ba
&\int_{\mathbb{T}^{3}} \nabla \varphi^\varep(\ell)\cdot \delta \bm{b}(\delta \bm{J}_T\cdot \delta \bm{b}_T)-\f{2}{|\bl|}\varphi^\varep\bn\cdot \delta\bm{b} (\delta\bm{J}_T\cdot \delta\bm{b}_T)+\f{1}{|\bl|}\bn\cdot\B[\delta \bm{b}(\delta\bm{b}\cdot \delta\bm{J})-\delta \bm{J}(\delta \bm{b}\cdot \delta \bm{b})\B] d^3\bl \\&+\s\B[\big(\bm{b} (\bm{J}_T\cdot \bm{b}_T)\big)^\varep-\bm{b}(\bm{J}_T\cdot \bm{b}_T)^\varep\B]\\
=&b_j\partial_k\B[(b_k J_{T_j})^\varep-(b_k J_{T_j}^\varep)\B]+J_i\partial_k\B[(b_k b_{T_i})^\varep-(b_k b_{T_i}^\varep)\B],
\ea\ee
and
\be\label{3.17}\ba
&\f12\int_{\mathbb{T}^{3}} \nabla \varphi^\varep(\ell)\cdot \delta \bm{J}[\delta \bm{b}_T]^2-\f{2}{|\bl|}\varphi^\varep\bn\cdot \delta \bm{J} [\delta \bm{b}_T]^2  -\f{2}{|\bl|}\bn\cdot\B[\delta \bm{b}(\delta \bm{J}\cdot \delta \bm{b})-\delta \bm{J}(\delta \bm{b}\cdot \delta \bm{b})\B]d^3\bl\\ &+\f12\s \B[\big(\bm{J}(\bm{b}_T\cdot\bm{b}_T)\big)^\varep-\bm{J}(\bm{b}_T\cdot\bm{b}_T)^\varep\B]
=b_j\partial_k\B[(J_k b_{T_j})^\varep-(J_k b_{T_j}^\varep)\B].
\ea\ee
Substituting \eqref{3.16} and \eqref{3.17} into \eqref{3.15}, we see that
\be\ba\label{3.19}
 &  \partial_{t}(\bm{b}_{T}^{\varepsilon}\cdot\bm{b})+\f{\text{d}_{\text{I}}}{2}\s([\s(\bm{b} \otimes \bm{b}_{T})]^{\varepsilon}\times \bm{b})+\f{\text{d}_{\text{I}}}{2}\s([\s(\bm{b} \otimes \bm{b})]\times \bm{b}_{T}^{\varepsilon})\\&+\f{\text{d}_{\text{I}}}{2}\s\big[\bm{b} (\bm{b}\cdot \bm{J}_{T}^{\varepsilon}) +\bm{b} (\bm{b}_{T}^{\varepsilon}\cdot \bm{J})- \bm{J} (\bm{b}\cdot \bm{b}_{T}^{\varepsilon})\big]
 \\&- \f{\text{d}_{\text{I}}}{4}\s \B[\big(\bm{J}(\bm{b}_T\cdot\bm{b}_T)\big)^\varep-\bm{J}(\bm{b}_T\cdot\bm{b}_T)^\varep\B]
+\f{\text{d}_{\text{I}}}{2} \s\B[\big(\bm{b} (\bm{J}_T\cdot \bm{b}_T)\big)^\varep-\bm{b}(\bm{J}_T\cdot \bm{b}_T)^\varep\B]\\
=&-\f43D_{ET}^{\varepsilon}(\bm{b},\bm{J}).
\ea\ee
where
\begin{align}
&D_{ET}^{\varepsilon}(\bm{b},\bm{J})\nonumber\\
=&-\f{3\text{d}_{\text{I}}}{16}\int_{\mathbb{T}^{3}} \nabla \varphi^\varep(\ell)\cdot \delta \bm{J}[\delta \bm{b}_T]^2-\f{2}{|\bl|}\varphi^\varep\bn\cdot \delta \bm{J} [\delta \bm{b}_T]^2  d^3\bl\\  &+\f{3\text{d}_{\text{I}}}{8}      \int_{\mathbb{T}^{3}} \nabla \varphi^\varep(\ell)\cdot \delta \bm{b}(\delta \bm{J}_T\cdot \delta \bm{b}_T)-\f{2}{|\bl|}\varphi^\varep\bn\cdot \B[\delta\bm{b} (\delta\bm{J}_T\cdot \delta\bm{b}_T)+\delta \bm{b}\times(\delta \bm{J}\times\delta \bm{b})\B] d^3\bl.\nonumber
\end{align}

With the help of \eqref{1/3}, we deduce from a
change of variables that
$$\ba
\B\|\bm{b}_{L}^{\varepsilon}-\f13  \bm{b}\B\|_{L^{p}(0,T;L^{q}(\mathbb{T}^{3}))}=&
\B\|\int_{\mathbb{T}^{3}} \varphi^\varep(\ell)(\bn\otimes \bn)\cdot\big[\bm{b}(\bx+\bl,t)-\bm{b}(\bx,t)\big]d^3\bl\B\|_{L^{p}(0,T;L^{q}(\mathbb{T}^{3}))}
\\\leq&\int_{\mathbb{T}^{3}} \varphi^\varep(\ell)\|\bm{b}(\bx+\bl,t)-\bm{b}(\bx,t)\|_{L^{p}(0,T;L^{q}(\mathbb{T}^{3}))}d^3\bl
\\=&\int_{\mathbb{T}^{3}} \varphi (\xi)\|\bm{b}(\bx+\varepsilon\xi,t)-\bm{b}(\bx,t)\|_{L^{p}(0,T;L^{q}(\mathbb{T}^{3}))}d^3\xi
\ea$$
The strong-continuity of translation operators on Lebesgue spaces ensures that
$$
\lim_{\varepsilon\rightarrow0}
\|\bm{u}(\bx+\varepsilon\xi,t)-\bm{u}(\bx,t)\|_{L^{p}(0,T;L^{q}(\mathbb{T}^{3}))}=0.
$$
Employing the dominated convergence theorem, we discover that
\be\label{0315}
\lim_{\varepsilon\rightarrow0}\B\|\bm{b}_{L}^{\varepsilon}-\f13  \bm{b}\B\|_{L^{p}(0,T;L^{q}(\mathbb{T}^{3}))}=0.
\ee

Hence, the left hand side of \eqref{3.14} convergences to
$$\f23 \B\{\partial_{t}(\f12|\bm{b}|^{2})  +\f{\text{d}_{\text{I}}}{2}\s\B([\s(\bm{b} \otimes \bm{b})\times \bm{b}]\B)
-\f{\text{d}_{\text{I}}}{4}\s (\bm{j}|\bm{b}|^{2})+\f{\text{d}_{\text{I}}}{2}\s[\bm{b}(\bm{j}\cdot \bm{b})]\B\}$$
in the sense of distribution.

Taking into consideration $\bm{b}_T^\varep=\bm{b}^\varep-\bm{b}_L^\varep$, we conclude by a slight variant of the proof of \eqref{0315}  that
\be\label{0316}
\lim_{\varepsilon\rightarrow0}\B\|\bm{b}_{T}^{\varepsilon}-\f23  \bm{b}\B\|_{L^{p}(0,T;L^{q}(\mathbb{T}^{3}))}=0.
\ee
As a consequence, the left hand side of \eqref{3.19} convergence  to
$$\f43 \B\{\partial_{t}(\f12|\bm{b}|^{2})  +\f{\text{d}_{\text{I}}}{2}\s\B([\s(\bm{b} \otimes \bm{b})\times \bm{b}]\B)
-\f{\text{d}_{\text{I}}}{4}\s (\bm{j}|\bm{b}|^{2})+\f{\text{d}_{\text{I}}}{2}\s[\bm{b}(\bm{j}\cdot \bm{b})]\B\}.
$$
Hence, the validity of \eqref{1.11} is confirmed.

To proceed further, we set
$$\ba
&\overline{S}_{EL}(\bm{b}, \bm{J},\lambda)=\f{\text{d}_{\text{I}}}{\lambda}\int_{\partial B } \bl \cdot \B[  \delta \bm{b}(\lambda\bl)(\delta \bm{b}_L(\lambda\bl)\cdot \delta \bm{J}_L(\lambda\bl))-\f12\delta \bm{J}(\lambda\bl)[\delta \bm{b}_L(\lambda\bl)]^2\B] \f{d\sigma(\bl) }{4\pi},\\
&\overline{S}_{ET} (\bm{b}, \bm{J},\lambda)=\f{\text{d}_{\text{I}}}{\lambda}\int_{\partial B } \bl \cdot \B[  \delta \bm{b}(\lambda\bl)(\delta \bm{b}_T(\lambda\bl)\cdot \delta \bm{J}_T(\lambda\bl))-\f12\delta \bm{J}(\lambda\bl)[\delta \bm{b}_T(\lambda\bl)]^2\B] \f{d\sigma(\bl) }{4\pi},\\
&\overline{S}_{E} (\bm{b}, \bm{J},\lambda)=\f{\text{d}_{\text{I}}}{\lambda}\int_{\partial B }  \bl \cdot\B[\delta\bm{b}(\lambda\bl)\times \B(\delta \bm{J}(\lambda\bl) \times \delta \bm{b}(\lambda\bl)\B)\B] \f{d\sigma(\bl) }{4\pi}.
\ea$$
By coarea formula and the change of variable, we find
\be\ba\label{3.21}
&D_{E}(\bm{b},\bm{J})\\=&\lim_{\varepsilon\rightarrow0}D_{ET}^{\varepsilon}(\bm{b},\bm{J})\\
=&\f{3\text{d}_{\text{I}}}{8}      \lim_{\varepsilon\rightarrow0}\int_{\mathbb{T}^{3}} (\nabla \varphi^\varep(\ell)-\f{2}{|\bl|}\varphi^\varep\bn)\cdot\big[ \delta \bm{b}(\bl)(\delta \bm{J}_T\cdot \delta \bm{b}_T)-\f12\delta \bm{J}(\bl)[\delta \bm{b}_T]^2\big]\\&-\f{3\text{d}_{\text{I}}}{4} \lim_{\varepsilon\rightarrow0}     \int_{\mathbb{T}^{3}}\f{1}{|\bl|}\bn\cdot\big[\delta \bm{b}\times(\delta \bm{J}\times\delta\bm{b})\big]d^3\bl\\
=&
\f{3}{8}\B(\int_0^\infty r^3\varphi^{'}(r)-2r^2\varphi(r)dr\B)4\pi \bar{S}_{ET}(\bm{b},\bm{J})-\f{3}{4}\int_0^\infty r^2\varphi(r)dr 4\pi\bar{S}_{E}(\bm{b},\bm{J})\\
=&-\f{15 }{8}\B(\bar{S}_{ET}(\bm{b},\bm{J})+\f25\bar{S}_{E}(\bm{b},\bm{J})\B)\\
=&-\f{15 }{8} S_{ET}(\bm{b},\bm{J}).
\ea\ee
It follows from \eqref{3.21}  that
$$S_{ET}(\bm{b},\bm{J})= -\f{8}{15} D_{E}(\bm{b},\bm{J}).$$
Moreover, an argument similar to the one used in \eqref{3.21}  shows that
\begin{align}
&D_{E}(\bm{b},\bm{J})\nonumber\\=&\lim_{\varepsilon\rightarrow0}D_{EL}^{\varepsilon}(\bm{b},\bm{J})\\
=&\f{3\text{d}_{\text{I}}}{4}\lim_{\varepsilon\rightarrow0} \int_{\mathbb{T}^{3}} \nabla \varphi^\varep(\ell)\cdot \B[\delta \bm{b}(\delta \bm{J}_L \cdot \delta \bm{b}_L)-\f12 \delta \bm{J} [\delta \bm{b}_L]^2\B]\nonumber\\&+\f{3\text{d}_{\text{I}}}{4}\lim_{\varepsilon\rightarrow0} \int_{\mathbb{T}^{3}}\f{2}{|\bl|}\varphi^\varep\bn\cdot[ \delta \bm{b}(\delta \bm{J}_T\cdot \bm{b}_T)-\f12 \delta \bm{J} [\delta \bm{b}_L]^2]  d^3 \bl+\f32\lim_{\varepsilon\rightarrow0} \int_{\mathbb{T}^{3}}\f{1}{|\bl|}\varphi^\varep\bn\cdot[ \delta\bm{b} \times(\delta\bm{J}\times\delta\bm{b})]  d^3 \bl\nonumber\\=&\f34\int_0^\infty r^3\varphi^{'}(r)dr 4\pi \bar {S}_{EL}(\bm{b},\bm{J})+\f34 \int_0^\infty 2\varphi(r) r^2 dr 4\pi \bar{S}_{ET}(\bm{b},\bm{J})+\f32\int_0^\infty r^2\varphi(r) dr 4\pi \bar{S}_{E}(\bm{b},\bm{J})\nonumber\\
=&-\f94 \bar {S}_{EL}(\bm{b},\bm{J})+\f32 \bar{S}_{ET}(\bm{b},\bm{J})+\f32\bar{S}_{E}(\bm{b},\bm{J}).\label{3.22}
\end{align}
A combination of \eqref{3.21} and \eqref{3.22}, we end up with
\be\label{3.23}\left\{\ba
&D_{E}(\bm{b},\bm{J})=-\f{15}{8} \bar{S}_{ET}(\bm{b},\bm{J})-\f34\bar{ S}_{E}(\bm{b},\bm{J}), \\
&D_{E}(\bm{b},\bm{J})=-\f94 \bar {S}_{EL}(\bm{b},\bm{J})+\f32 \bar{S}_{ET}(\bm{b},\bm{J})+\f32\bar{S}_{E}(\bm{b},\bm{J}).
 \ea\right.\ee
As a consequence, we  have
\be\label{d5}\ba
D_{E}(\bm{b},\bm{J})=-\f54\bar{S}_{EL}(\bm{b},\bm{J})+\f12\bar{S}_{E}(\bm{b},\bm{J})=-\f54 S_{EL}(\bm{b},\bm{J}),
\ea\ee
which means we have the following 4/5 law and 8/15 law of  energy for  EMHD equations
$$S_{EL}=-4/5 D_E(\bm{b},\bm{J})~~\text{and}~~S_{ET}=-8/15 D_E(\bm{b},\bm{J}),$$
where
\begin{equation}
	\ba
S_{EL}=&\lim_{\varepsilon\rightarrow0}\f{\text{d}_{\text{I}}}{\lambda}\int_{\partial B } \bl \cdot \B[  \delta \bm{b}(\lambda\bl)(\delta \bm{b}_L(\lambda\bl)\cdot \delta \bm{J}_L(\lambda\bl))-\f12\delta \bm{J}(\lambda\bl)[\delta \bm{b}_L(\lambda\bl)]^2\B] \f{d\sigma(\bl) }{4\pi}\\
&-\f25\lim_{\varepsilon\rightarrow0}\f{\text{d}_{\text{I}}}{\lambda}	\int_{\partial B }  \bl \cdot\B[\delta\bm{b}(\lambda\bl)\times \B(\delta \bm{J}(\lambda\bl) \times \delta \bm{b}(\lambda\bl)\B)\B] \f{d\sigma(\bl) }{4\pi},
	\ea
\end{equation}
and
\begin{equation}
	\ba
S_{ET}=&	\lim_{\varepsilon\rightarrow0}\f{\text{d}_{\text{I}}}{\lambda}\int_{\partial B } \bl \cdot \B[  \delta \bm{b}(\lambda\bl)(\delta \bm{b}_T(\lambda\bl)\cdot \delta \bm{J}_T(\lambda\bl))-\f12\delta \bm{J}(\lambda\bl)[\delta \bm{b}_T(\lambda\bl)]^2\B] \f{d\sigma(\bl) }{4\pi}\\
&+\f25 \lim_{\varepsilon\rightarrow0}\f{\text{d}_{\text{I}}}{\lambda}\int_{\partial B }  \bl \cdot\B[\delta\bm{b}(\lambda\bl)\times \B(\delta \bm{J}(\lambda\bl) \times \delta \bm{b}(\lambda\bl)\B)\B] \f{d\sigma(\bl) }{4\pi}.
	\ea
\end{equation}
The proof of this theorem is finished.
 \end{proof}

\section{Exact relation of magnetic helicity in the inviscid  Hall-MHD equations}
This section is devoted to  establishing the scaling laws of magnetic helicity in the inviscid Hall-MHD equations.

\begin{proof}[Proof of Theorem  \ref{the1.3}]
Owing to $\eqref{VI}_2$ and  Putting \eqref{hallMHD} and  \eqref{hmpotentialeq} together, we infer that
 \be \label{4.1}\left\{\ba
&\bm{b}_{t}  +\nabla\times(\bm{b}\times \bm{u}) +\text{d}_{\text{I}}\nabla\times[\s(\bm{b} \otimes \bm{b} )] =0, \\
&\bm{A}_{t} -(\bm{u}\times \bm{b}) +\text{d}_{\text{I}}\s(\bm{b} \otimes \bm{b})  + \nabla  \pi=0.
 \ea\right.\ee
  According to \eqref{defiofeyink}, we notice that
\be \label{4.2}\left\{\ba
&\partial_{t}\bm{b}_{X}^{\varepsilon}  +\nabla\times(\bm{b}\times \bm{u}_{X})^{\varepsilon}+\text{d}_{\text{I}}\nabla\times[\s(\bm{b} \otimes \bm{b} _{X})]^{\varepsilon} =0, \\
&\partial_{t}\bm{A}_{X}^{\varepsilon} -(\bm{u}\times \bm{b}_{X})^{\varepsilon}+\text{d}_{\text{I}}\s(\bm{b} \otimes \bm{b}_{X})^{\varepsilon}  +\s  {\bf \Pi}^{\varepsilon}_{X}=0,~\text{with}~X=L,T.
 \ea\right.\ee
 where
\be\label{c7}
{\bf \Pi}^\varep_L=\int_{\mathbb{T}^{3}}  \varphi^\varep(\bl)(\bn\otimes \bn)\pi(\bx+\bl,t)d^3\bl,~\text{and}~{\bf \Pi}^\varep_T=\int_{\mathbb{T}^{3}}  \varphi^\varep(\bl)(\bm{1}-\bn\otimes \bn)\pi(\bx+\bl,t)d^3\bl.
\ee
Before proceeding further, we  claim that
\be\label{c8}
\s {\bf \Pi}^\varep_X(\bx,t)=\nabla \pi_X(\bx,t),~~X=L,T,
\ee
holds in the sense of distributions, where $\pi_L(\bx,t), \pi_T(\bx,t)$ are scalar functions defined  as follows
\be\label{c9}
\pi_X^\varep(\bx,t)=\int_{\mathbb{T}^{3}}  \varphi^\varep_X(\ell)\pi(\bx+\bl,t)d^3\bl,~~X=L,T,
\ee
with
\be\label{c10}
\varphi_L(\ell)=\varphi(\ell)-\varphi_T(\ell),~~~~\varphi_T(\ell)=2\int_{|\bl|}^\infty \f{\varphi(\ell^{'})}{\ell^{'}} d\ell^{'}.
\ee
With the help of  the definition of $\varphi(\ell)$, it is easy to see that $\varphi_L(\ell)$ and $ \varphi_T(\ell)$ defined here  are compactly supported and $C^\infty$ everywhere except at $0$, where they have a mild (logarithmic) singularity.

Now we are in a position to  show the validity of \eqref{c8}. On the one hand, using a straightforward computation, we have the following basic relation
\be\label{c11}\ba
\partial_{\ell_k}\bn_i&=\f{\partial}{\partial\ell_k}\B(\f{\ell_i}{|\bl|}\B)=\f{\delta^{ik}|\bl|-\f{\ell_i\ell_k}{|\bl|}}{|\bl]^2}=\f{1}{|\bl|}(\delta^{ik}-\bn_i\bn_k),
\ea\ee
which together with the integration by parts yields
\be\label{c12}\ba
\s {\bf \Pi}^\varep_L&=\partial_{x_i}\int_{\mathbb{T}^{3}}  \varphi^\varep(\ell)(\bn_i\bn_j)\pi(\bx+\bl) d^3\bl\\
&=\int_{\mathbb{T}^{3}}  \varphi^\varep(\ell)(\bn_i\bn_j)\partial_{\ell_i}\pi(\bx+\bl) d^3\bl\\
&=-\int_{\mathbb{T}^{3}}  \B[\f{d\varphi^\varep(\ell)}{d\ell}\bn_j+\varphi^\varep(\ell)\f{2}{|\bl|}\bn_j\B]\pi(\bx+\bl) d^3\bl.
\ea\ee
 On ther other hand, in light of the  definition of  $\Pi_L^\varep$ and \eqref{c10}, we have
\be\label{c13}\ba
\nabla \varphi^\varep_L(\ell)=\B(\f{d\varphi^\varep}{d\ell}(\ell)+\f{2}{|\bl|}\varphi^\varep(\ell)\B)\bn.
\ea\ee
Combining \eqref{c12}, \eqref{c13} and using the integration by parts again, we can deduce that
\be\label{c14}\ba
\s {\bf \Pi}^\varep_L&=-\int_{\mathbb{T}^{3}}  \partial_{\ell_j}\varphi^\varep_L(\ell)\pi(\bx+\bl,t)d^3\ell\\
&=\int_{\mathbb{T}^{3}} \varphi^\varep_L(\ell)\partial_{x_j}\pi(\bx+\bl,t)d^3\ell\\
&=\partial_{x_j}\int_{\mathbb{T}^{3}} \varphi^\varep_L(\ell)\pi(\bx+\bl,t)d^3\ell=\nabla \pi_L^\varep(\bx).\\
\ea\ee
Similarly, since ${\bf \Pi}_L^\varep+{\bf \Pi}_T^\varep=\pi^\varep{\bf 1}$ and $\pi_L^\varep+\pi_T^\varep=\pi^\varep$, we can also obtain $\s {\bf \Pi}_T^\varep=\nabla \pi_T^\varep$. Then we have proved the claim \eqref{c8}.

The next thing to do in the proof is trying to establish the local   longitudinal and transverse K\'arm\'arth-Howarth equations for magnetic helicity. To this end, letting $X=L$ and dotting $\eqref{4.2}_1$ and $\eqref{4.2}_2$ by $\bm{A}$ and $\bm{b}$, respectively,  we have
\be\label{4.3} \left\{\ba
&\partial_{t}\bm{b}_{L}^{\varepsilon}\cdot\bm{A}  +\nabla\times(\bm{b}\times \bm{u}_{L})^{\varepsilon}\cdot\bm{A} +\nabla\times[\s(\bm{b} \otimes \bm{b} _{L})]^{\varepsilon} \cdot\bm{A}=0, \\
&\partial_{t}\bm{A}_{L}^{\varepsilon}\cdot \bm{b}-(\bm{u}\times \bm{b}_{L})^{\varepsilon}\cdot \bm{b}+\text{d}_{\text{I}}\s(\bm{b} \otimes \bm{b}_{L})^{\varepsilon} \cdot \bm{b} +\bm{b}\cdot\nabla  \pi^{\varepsilon}_{L}=0,
 \ea\right.\ee
Multiplying $\eqref{4.1}_1$ and  $\eqref{4.1}_2$ with $X=L$ by $\bm{A}_{L}^{\varepsilon}$ and $\bm{b}_{L}^{\varepsilon}$, respectively, we can obtain
\be \label{4.4}\left\{\ba
&\bm{b}_{t}\cdot\bm{A}_{L}^{\varepsilon} +\nabla\times(\bm{b}\times \bm{u})\cdot\bm{A}_{L}^{\varepsilon}+\text{d}_{\text{I}}\nabla\times[\s(\bm{b} \otimes \bm{b} )] \cdot\bm{A}_{L}^{\varepsilon}=0, \\
&\bm{A}_{t}\cdot\bm{b}_{L}^{\varepsilon}-(\bm{u}\times \bm{b})\cdot\bm{b}_{L}^{\varepsilon}+\text{d}_{\text{I}}\s(\bm{b} \otimes \bm{b})\cdot\bm{b}_{L}^{\varepsilon}  +\bm{b}_{L}^{\varepsilon}\cdot\nabla  \pi=0,
 \ea\right.\ee
By the sum of \eqref{4.3} and \eqref{4.4}, we infer that
 \be\ba\label{4.5}
&\partial_t(\bm{A}_{L}^{\varepsilon}\bm{b}+\bm{A}\bm{b}_{L}^{\varepsilon}) -(\bm{u}\times \bm{b})\cdot\bm{b}_{L}^{\varepsilon}-(\bm{u}\times \bm{b}_{L})^{\varepsilon}\cdot\bm{b}-\nabla\times(\bm{u}\times \bm{b}_{L})^{\varepsilon}\cdot \bm{A}-\nabla\times(\bm{u}\times \bm{b})\cdot \bm{A}_{L}^{\varepsilon}\\&+
\text{d}_{\text{I}}\s(\bm{b}\otimes \bm{b}_{L})^{\varepsilon}\cdot\bm{b}+\text{d}_{\text{I}}\s(\bm{b}\otimes \bm{b})\cdot\bm{b}_{L}^{\varepsilon}  +\nabla \pi_{L}^{\varepsilon}\bm{b} \\&+\nabla \pi \bm{b}_{L}^{\varepsilon}+\text{d}_{\text{I}}\nabla\times[\s(\bm{b}\otimes \bm{b}_{L})]^{\varepsilon}\cdot\bm{A}+\text{d}_{\text{I}}\nabla\times[\s(\bm{b}\otimes \bm{b})]\cdot\bm{A}_{L}^{\varepsilon}=0.
\ea\ee
On the other hand, by virtue of  $ \eqref{VI}_{3}$, it follows that
\be\label{4.6}\ba
\bm{A}\cdot[\nabla\times(\bm{u}\times \bm{b}_{L})^{\varepsilon}] =&\s[(\bm{u}\times \bm{b}_{L})^{\varepsilon}\times \bm{A}]+(\bm{u}\times \bm{b}_{L})^{\varepsilon}\cdot(\nabla\times \bm{A})
\\=&\s[(\bm{u}\times \bm{b}_{L})^{\varepsilon}\times \bm{A}]+(\bm{u}\times \bm{b}_{L})^{\varepsilon}\cdot \bm{b}.
\ea\ee
and
\be\label{4.7}\ba
\bm{A}_{L}^{\varepsilon}\cdot[\nabla\times(\bm{u}\times \bm{b})]
 =&\s[(\bm{u}\times \bm{b})\times \bm{A}_{L}^{\varepsilon}]+(\bm{u}\times \bm{b})\cdot \bm{b}_{L}^{\varepsilon}.
\ea\ee
In the same manner as above, we obtain
\be\label{4.8}\ba
\bm{A}\cdot\{  \nabla\times[\s(\bm{b}\otimes \bm{b}_{L})]^{\varepsilon}\}=&\s([\s(\bm{b}\otimes \bm{b}_{L})]^{\varepsilon}\times \bm{A})+[\s(\bm{b}\otimes \bm{b}_{L})]^{\varepsilon}\cdot(\nabla\times  \bm{A} )
\\=&\s([\s(\bm{b}\otimes \bm{b}_{L})]^{\varepsilon}\times \bm{A})+[\s(\bm{b}\otimes \bm{b}_{L})]^{\varepsilon}\cdot \bm{b}
\ea\ee
and
\be\label{4.9}\ba
\bm{A}_{L}^{\varepsilon}\cdot\{  \nabla\times[\s(\bm{b}\otimes \bm{b})]\}=&\s([\s(\bm{b}\otimes \bm{b})]\times \bm{A}_{L}^{\varepsilon})+[\s(\bm{b}\otimes \bm{b})]\cdot(\nabla\times \bm{A}_{L}^{\varepsilon})
\\=&\s([\s(\bm{b}\otimes \bm{b})]\times \bm{A}_{L}^{\varepsilon})+[\s(\bm{b}\otimes \bm{b})]\cdot \bm{b}_{L}^{\varepsilon}.
\ea\ee
Plugging \eqref{4.6} and \eqref{4.9} into \eqref{4.5},
we observe that
\be\label{4.10}\ba
&\partial_t(\bm{A}_{L}^{\varepsilon}\cdot \bm{b}+\bm{A}\cdot \bm{b}_{L}^{\varepsilon})-\s[(\bm{u}\times \bm{b}_{L})^{\varepsilon}\times \bm{A}]-2(\bm{u}\times \bm{b}_{L})^{\varepsilon}\cdot \bm{b}-\s[(\bm{u}\times \bm{b})\times \bm{A}_{L}^{\varepsilon}]\\&-2(\bm{u}\times \bm{b})\cdot \bm{b}^{\varepsilon}_{L}+ \text{d}_{\text{I}}\s([\s(\bm{b}\otimes \bm{b}_{L})]^{\varepsilon}\times \bm{A})
 +\text{d}_{\text{I}}\s([\s(\bm{b}\otimes \bm{b})]\times \bm{A}_{L}^{\varepsilon}) +\s[ \pi_L^{\varepsilon}\bm{b}+ \pi \bm{b}_{L}^{\varepsilon}]\\
 =&-2\text{d}_{\text{I}}[ \s(\bm{b}\otimes \bm{b}_{L})^{\varepsilon}\bm{b}+\s(\bm{b}\otimes \bm{b})\bm{b}_{L}^{\varepsilon}].
\ea\ee
In view of Leibniz's formula, we deduce that
\be\ba\label{4.11}
\s(\bm{b}\otimes \bm{b}_{L})^{\varepsilon}\cdot\bm{b}+\s(\bm{b}\otimes \bm{b})\cdot\bm{b}_{L}^{\varepsilon}=&\partial_{k}( {b}_{k}  {b}_{L_{i}})^{\varepsilon} {b}_{i}+\partial_{k}( {b}_{k}  {b}_{i}) {b}_{L_{i}}^{\varepsilon}\\=&
\partial_{k}(b_{k}  b_{i}b_{L_{i}}^{\varepsilon})+\partial_{k}(b_{k}  b_{L_{i}})^{\varepsilon}b_{i}-(b_{k}  b_{i})\partial_{k}b_{L_{i}}^{\varepsilon}\\=&
\s\B[\bm{b}(\bm{b}\cdot\bm{b}_L^\varep)  \B]+b_{i}\partial_{k}\B[(b_{k}  b_{L_{i}})^{\varepsilon}- (b_{k} b_{L_{i}}^{\varepsilon})\B].
\ea\ee
Hence, by substituting \eqref{4.11} into \eqref{4.10}, we have
\be\label{4.12}\ba
&\partial_t(\bm{A}_{L}^{\varepsilon}\cdot \bm{b}+\bm{A}\cdot \bm{b}_{L}^{\varepsilon})-\s[(\bm{u}\times \bm{b}_{L})^{\varepsilon}\times \bm{A}]-2(\bm{u}\times \bm{b}_{L})^{\varepsilon}\cdot \bm{b}-\s[(\bm{u}\times \bm{b})\times \bm{A}_{L}^{\varepsilon}]\\&-2(\bm{u}\times \bm{b})\cdot \bm{b}^{\varepsilon}_{L}+ \text{d}_{\text{I}}\s([\s(\bm{b}\otimes \bm{b}_{L})]^{\varepsilon}\times \bm{A})
+\text{d}_{\text{I}}\s([\s(\bm{b}\otimes \bm{b})]\times \bm{A}_{L}^{\varepsilon})\\
& +\s[ \pi_L^{\varepsilon}\bm{b}+ \pi \bm{b}_{L}^{\varepsilon}]+2\text{d}_{\text{I}}\s\B[\bm{b}(\bm{b}\cdot\bm{b}_L^\varep)  \B]\\
=&-2\text{d}_{\text{I}}b_{i}\partial_{k}\B[(b_{k}  b_{L_{i}})^{\varepsilon}- (b_{k} b_{L_{i}}^{\varepsilon})\B].
\ea\ee
Employing \eqref{2.17} in Lemma \ref{lemma2.2}, we find
 \be\label{4.13}\ba
  &\int_{\mathbb{T}^{3}} \nabla \varphi^\varep(\ell)\cdot \delta \bm{b} [\delta \bm{b}_L]^2+\f{2}{|\bl|}\varphi^\varep\bn\cdot \delta \bm{b} [\delta \bm{b}_T]^2 d^3\bl
   \\& +\s \B[\big(\bm{b}(\bm{b}_L\cdot\bm{b}_L)\big)^\varep-\bm{b}(\bm{b}_L\cdot\bm{b}_L)^\varep\B]
  \\=&2b_j\partial_k\B[(b_k b_{L_j})^\varep-(b_k b_{L_j}^\varep)\B].
   \ea\ee
Combining \eqref{4.13} and \eqref{4.12}, we discover that
\be\ba\label{04.22}
&\partial_t(\bm{A}_{L}^{\varepsilon}\cdot \bm{b}+\bm{A}\cdot \bm{b}_{L}^{\varepsilon})-\s[(\bm{u}\times \bm{b}_{L})^{\varepsilon}\times \bm{A}]-2(\bm{u}\times \bm{b}_{L})^{\varepsilon}\cdot \bm{b}-\s[(\bm{u}\times \bm{b})\times \bm{A}_{L}^{\varepsilon}]\\&-2(\bm{u}\times \bm{b})\cdot \bm{b}^{\varepsilon}_{L}+ \text{d}_{\text{I}}\s([\s(\bm{b}\otimes \bm{b}_{L})]^{\varepsilon}\times \bm{A})
+\text{d}_{\text{I}}\s([\s(\bm{b}\otimes \bm{b})]\times \bm{A}_{L}^{\varepsilon})\\
& +\s[ \pi_L^{\varepsilon}\bm{b}+ \pi \bm{b}_{L}^{\varepsilon}]+2\text{d}_{\text{I}}\s\B[\bm{b}(\bm{b}\cdot\bm{b}_L^\varep)  \B]+\text{d}_{\text{I}}\s \B[\big(\bm{b}(\bm{b}_T\cdot\bm{b}_T)\big)^\varep-\bm{b}(\bm{b}_T\cdot\bm{b}_T)^\varep\B]
\\
   =&-\text{d}_{\text{I}}\int_{\mathbb{T}^{3}} \nabla \varphi^\varep(\ell)\cdot \delta \bm{b} [\delta \bm{b}_L]^2+\f{2}{|\bl|}\varphi^\varep\bn\cdot \delta \bm{b} [\delta \bm{b}_T]^2 d^3\bl\\
   =&-\f43 D_{ML}^{\varepsilon}(\bm{b},\bm{b}),
\ea\ee
where
$$
D_{ML}^{\varepsilon}(\bm{b},\bm{b})=\f34\text{d}_{\text{I}}\int_{\mathbb{T}^{3}} \nabla \varphi^\varep(\ell)\cdot \delta \bm{b} [\delta \bm{b}_L]^2+\f{2}{|\bl|}\varphi^\varep\bn\cdot \delta \bm{b} [\delta \bm{b}_T]^2 d^3\bl.
$$
Along the exact same lines as \eqref{4.12},  we have
\be\label{4.14}\ba
&\partial_t(\bm{A}_{T}^{\varepsilon}\cdot \bm{b}+\bm{A}\cdot \bm{b}_{T}^{\varepsilon})-\s[(\bm{u}\times \bm{b}_{T})^{\varepsilon}\times \bm{A}]-2(\bm{u}\times \bm{b}_{T})^{\varepsilon}\cdot \bm{b}-\s[(\bm{u}\times \bm{b})\times \bm{A}_{T}^{\varepsilon}]\\&-2(\bm{u}\times \bm{b})\cdot \bm{b}^{\varepsilon}_{T}+ \text{d}_{\text{I}}\s([\s(\bm{b}\otimes \bm{b}_{T})]^{\varepsilon}\times \bm{A})
+\text{d}_{\text{I}}\s([\s(\bm{b}\otimes \bm{b})]\times \bm{A}_{T}^{\varepsilon})\\
& +\s[ \pi_T^{\varepsilon}\bm{b}+ \pi \bm{b}_{T}^{\varepsilon}]+2\text{d}_{\text{I}}\s\B[\bm{b}(\bm{b}\cdot\bm{b}_T^\varep)  \B]\\
=&-2\text{d}_{\text{I}}b_{i}\partial_{k}\B[(b_{k}  b_{T_{i}})^{\varepsilon}- (b_{k} b_{T_{i}}^{\varepsilon})\B].
\ea\ee
by virtue of \eqref{2.19} in Lemma \ref{lemma2.2}, we arrive at
\be\label{4.15}\ba
& \int_{\mathbb{T}^{3}} \nabla \varphi^\varep(\ell)\cdot \delta \bm{b}[\delta \bm{b}_T]^2-\f{2}{|\bl|}\varphi^\varep\bn\cdot \delta \bm{b} [\delta \bm{b}_T]^2  d^3\bl + \s \B[\big(\bm{b}(\bm{b}_T\cdot\bm{b}_T)\big)^\varep- \bm{b}(\bm{b}_T\cdot\bm{b}_T)^\varep\B]\\
&=2b_j\partial_k\B[(b_k b_{T_j})^\varep-(b_k b_{T_j}^\varep)\B].
\ea\ee
A combination of \eqref{4.14} and \eqref{4.15}, we know that
\be\ba\label{04.25}
&\partial_t(\bm{A}_{T}^{\varepsilon}\cdot \bm{b}+\bm{A}\cdot \bm{b}_{T}^{\varepsilon})-\s[(\bm{u}\times \bm{b}_{T})^{\varepsilon}\times \bm{A}]-2(\bm{u}\times \bm{b}_{T})^{\varepsilon}\cdot \bm{b}-\s[(\bm{u}\times \bm{b})\times \bm{A}_{T}^{\varepsilon}]\\&-2(\bm{u}\times \bm{b})\cdot \bm{b}^{\varepsilon}_{T}+ \text{d}_{\text{I}}\s([\s(\bm{b}\otimes \bm{b}_{T})]^{\varepsilon}\times \bm{A})
+\text{d}_{\text{I}}\s([\s(\bm{b}\otimes \bm{b})]\times \bm{A}_{T}^{\varepsilon})\\
& +\s[ \pi_T^{\varepsilon}\bm{b}+ \pi \bm{b}_{T}^{\varepsilon}]+2\text{d}_{\text{I}}\s\B[\bm{b}(\bm{b}\cdot\bm{b}_T^\varep)  \B]+ \s \B[\big(\bm{b}(\bm{b}_T\cdot\bm{b}_T)\big)^\varep- \bm{b}(\bm{b}_T\cdot\bm{b}_T)^\varep\B]\\
     =&-\f83D_{MT}^{\varepsilon}(\bm{b},\bm{b}),
\ea\ee
where
$$\ba
D_{MT}^{\varepsilon}(\bm{b},\bm{b})
=&\f38\text{d}_{\text{I}}\int_{\mathbb{T}^{3}}\B[\nabla \varphi^\varep(\ell)-\f{2}{|\bl|}\varphi^\varep\bn\B]\cdot  \delta \bm{b} [\delta \bm{b}_T]^2  d^3\bl.
\ea$$

With \eqref{04.22} and \eqref{04.25} in hand,
mimicking the proof of  \eqref{1.11}, we get
\eqref{01.20}. Next, we write
$$\ba
& {S}_{ML}(\bm{b}, \bm{b},\lambda)=\f{\text{d}_{\text{I}}}{\lambda}\int_{\partial B } \bn \cdot \B[    \delta \bm{b} (\lambda\bl)[\delta \bm{b}_L(\lambda\bl)]^2 \B] \f{d\sigma(\bl) }{4\pi},\\
& {S}_{MT}(\bm{b}, \bm{b},\lambda)=\f{\text{d}_{\text{I}}}{\lambda}\int_{\partial B } \bn \cdot \B[    \delta \bm{b} (\lambda\bl)[\delta \bm{b}_T(\lambda\bl)]^2 \B] \f{d\sigma(\bl) }{4\pi}.
\ea$$
By polar coordinates and the change of variable,  we find
\be\ba\label{4.16}
D_{M}(\bm{b},\bm{b})=&\lim_{\varepsilon\rightarrow0}
D^{\varepsilon}_{ET}(\bm{b},\bm{b})\\
=&\lim_{\varepsilon\rightarrow0}\f38\B(\int_0^\infty r^3\varphi^{'}(r)-2r^2\varphi(r)dr\B)4\pi {S}_{MT}(\bm{b},\bm{b},\varepsilon\bl)\\
  =&-\f{15}{8}  {S}_{MT}(\bm{b},\bm{b}).
\ea\ee
Likewise,
$$\ba
D_{ML}(\bm{b},\bm{b})=&
\lim_{\varepsilon\rightarrow0}D_{ML}^{\varepsilon}(\bm{b},\bm{b})
\\
 =&\f34\text{d}_{\text{I}}\lim_{\varepsilon\rightarrow0}\int_{\mathbb{T}^{3}} \nabla \varphi^\varep(\ell)\cdot  \delta \bm{b} [\delta \bm{b}_L]^2  +\f34\text{d}_{\text{I}}\lim_{\varepsilon\rightarrow0}\int_{\mathbb{T}^{3}}\f{2}{|\bl|}\varphi^\varep\bn\cdot \delta \bm{b} [\delta \bm{b}_T]^2  d^3\bl
 \\=&\f34\int_0^\infty r^3\varphi^{'}(r)dr 4\pi   {S}_{ML}(\bm{b},\bm{b},\varepsilon\bl)+\f34 \lim_{\varepsilon\rightarrow0}\int_0^\infty 2\varphi(r) r^2 dr 4\pi  {S}_{MT} (\bm{b},\bm{b},\varepsilon\bl)\\
=&-\f94   {S}_{ML}(\bm{b},\bm{b})+\f32  {S}_{MT}(\bm{b},\bm{b}).
  \ea$$
This together with \eqref{4.16} implies that
  \be\label{4.17}\left\{\ba
&D_{M}(\bm{b},\bm{b})=-\f{15}{8}  {S}_{MT}(\bm{b},\bm{b}) ,\\
&D_{M}(\bm{b},\bm{b})=-\f94  {S}_{ML}(\bm{b},\bm{b})+
\f32 {S}_{MT}(\bm{b},\bm{b}),
  \ea\right.\ee
which turns out that
$$D_{M}(\bm{b},\bm{b})=-\f54 {S}_{ML}(\bm{b},\bm{b}).$$
Moreover,  from \eqref{4.16}, we deduce that
$$D_{M}(\bm{b},\bm{b})=-\f{8}{15} {S}_{MT}(\bm{b},\bm{b}).$$
This completes the proof of the scaling law of this theorem.
\end{proof}

\section{Concluding remarks}
The Kolmogorov law in \cite{[Kolmogorov]} and Yaglom law  in \cite{[Yaglom]} are rare rigorous results in turbulence. The exact scaling laws of conserved quantities such as energy, cross helicity and magnetic helicity  in physical spaces paly an important  role in the study of the plasma turbulence.
EMHD and Hall-MHD equations are more suitable than the standard MHD equations on scales below the
ion inertial length. The Yaglom type law of the Hall-MHD equations was found by Hellinger,  Verdini,  Landi, Franci  and  Matteini  in \cite{[HVLFM]} (see also \cite{[FGSMA],[WC]}). For the Kolmogorov type 4/5 law, it is shown that this kind law of magnetic
helicity in the EMHD equations by Chkhetiani in \cite{[Chkhetiani]}.

The purpose of the current paper is to consider the 4/5 laws of energy, magnetic
helicity and  generalized helicity    in the EMHD and Hall MHD equations. Indeed, the Kolmogorov type law of energy and cross helicity in  the classical MHD equations were recently derived in \cite{[YWC]}. The proof relies on
  Eyink's  velocity decomposition in \cite{[Eyink1]} and
the    analysis of the  interaction of different physical quantities. The nonlinear term of the EMHD equations is the Hall-term based on second order derivative rather than the   convection term in terms of first order derivative in the traditional MHD equations, which brings more difficulties.
 In the spirit of work \cite{[YWC],[Eyink1]}, making full use of structure of  the Hall term, for the energy in the  EMHD
 equations, we have
$$S_{EL} (\bm{b},\bm{J})= -\f45 D_{E}(\bm{b},\bm{J}),$$
which corresponds  to
$$\langle\delta\bm{b}_{L} (\delta\bm{J}_{L}\cdot \delta\bm{b}_{L})\rangle-\f12\langle\delta\bm{J}_{L} [\delta\bm{b}_{L}|^{2}\rangle-\f25\langle\delta\bm{J}_{L} [\delta\bm{b} |^{2}\rangle+\f25\langle\delta\bm{b}_{L}  (\delta\bm{J} \cdot \delta\bm{b} )\rangle=-\f45\epsilon_{E} \bm{r}.$$
  As a byproduct of this and the scaling law of energy in the MHD equations obtained in \cite{[YWC]}, we get
   $$\ba
\langle \delta \bm{u}_{L}  [\delta \bm{u}_{L}|^{2}  \rangle &+  \langle \delta \bm{u}_{L} [\delta \bm{h}_{L} |^{ 2}  \rangle  -2\langle \delta \bm{h}_{L} (\delta \bm{h}_{L} \cdot\delta \bm{u}_{L} )
\rangle-\f45\langle \delta \bm{h}_{L}(\delta \bm{h} \cdot\delta \bm{v} )  \rangle+\f45\langle \delta \bm{v}_{L}[\delta \bm{h} | ^{2}  \rangle\\&+\langle\delta\bm{b}_{L} (\delta\bm{J}_{L}\cdot \delta\bm{b}_{L})\rangle-\f12\langle\delta\bm{J}_{L} [\delta\bm{b}_{L}|^{2}\rangle-\f25\langle\delta\bm{J}_{L} [\delta\bm{b} |^{2}\rangle+\f25\langle\delta\bm{b}_{L}  (\delta\bm{J} \cdot \delta\bm{b} )\rangle=-\f45\epsilon_{E}\bm{r}.
\ea$$
  For the magnetic helicity of the Hall-MHD equations, there holds
   $$\langle[ \delta \bm{u}_{L}(\bm{r})]^{3} \rangle=-\f45\epsilon_{M}\bm{r}.$$
%  Though the magnetic helicity is the conserved quantities of the MHD and the Hall-MHD,
A  similar argument      shows that, for the generalized helicity,
$$\ba
&\langle \delta \bv_{L} (\delta \bv_{L} \cdot\delta \bomega_{L} )  \rangle- \f12\langle \delta \bomega_{L} (\delta \bv_{L} )^{ 2}  \rangle  +2\langle \delta \bv_{L} (\delta \bm{h}_{L} \cdot\delta \bv_{L} )
\rangle -\langle \delta \bm{h}_{L}  (\delta \bv_{L})^{2} \rangle \\&-\f25\langle  \delta\bomega_{L}(\delta\bv)^{2}\rangle+ \f25\langle  \delta\bv_{L}(\delta\bv\cdot\delta\bomega)\rangle -\f45\langle \delta \bm{h}_{L} ( \delta \bv )^{2}
\rangle+\f45\langle \delta \bm{v}_{L} ( \delta \bv\cdot \delta \bm{h}  )
\rangle=-\f45\epsilon_{H} \bm{r}.
\ea$$

 Finally, we would like to summary the scaling laws of conserved  quantity in hydrodynamic fluids and plasma fluids \cite{
[Chkhetiani2],[Eyink1],[GPP], [HVLFM],[YWC],[Yaglom],[PP1],[MY],[Kolmogorov]} as follows.
 \renewcommand{\arraystretch}{1.7}
 \begin{table}[H] \centering
\begin{tabular}{|c|c|c|c|c|c|}

 \hline
\diagbox{conserved}{Model}&Euler	&MHD	&EMHD	&HMHD \\ \hline
Energy&	 $ \frac{4}{5}, \frac43, \f{8}{15} $  &	  $\frac{4}{5}, \f43, \f{8}{15} $  	&  $\frac{4}{5}, \f43, \f{8}{15} $ &$\frac{4}{5}, \f43, \f{8}{15} $	\\ \hline
(Cross) Helicity	&$\frac{4}{5}, \f43, \f{8}{15}, \f{2}{15}$	&$\frac{4}{5}, \f43, \f{8}{15} $&	$-\hspace{-0.2cm}-\hspace{-0.3cm}-\hspace{-0.3cm}-$	&	$-\hspace{-0.2cm}-\hspace{-0.3cm}-\hspace{-0.3cm}-$	 \\ \hline
Magnetic helicity	&$-\hspace{-0.2cm}-\hspace{-0.3cm}-\hspace{-0.3cm}-$&	$-\hspace{-0.2cm}-\hspace{-0.3cm}-\hspace{-0.3cm}-$&	$\frac{4}{5}, \f43, \f{8}{15}, \f{2}{15}$	&$\frac{4}{5}, \f43, \f{8}{15} $ \\ \hline
	\end{tabular}
\caption{Exact laws in   fluids}
 \end{table}

 \section*{Acknowledgement}
 Wang was partially supported by  the National Natural
 Science Foundation of China under grant (No. 11971446, No. 12071113   and  No.  11601492) and  sponsored by Natural Science Foundation of Henan (No. 232300421077). Ye was partially sponsored by National Natural Science Foundation of China (No. 11701145)
and Natural Science Foundation of Henan (No. 232300420111).

\end{document}